\documentclass[graybox]{svmult}
\usepackage{graphicx} 
\usepackage{amssymb,amsthm,amsmath} 
\usepackage[margin=1in]{geometry} 
\usepackage{fancyhdr} 
\usepackage{colonequals}
\usepackage{amsmath, amsthm}%
\usepackage{amsfonts}%
\usepackage{amssymb}%
\usepackage{amscd, amstext, hyperref}
\usepackage{pictexwd}
\usepackage{dcpic}
\usepackage[mathscr]{eucal}
\usepackage{multicol}
\usepackage{multirow}
\usepackage{color}
\usepackage{bm}
\usepackage{comment}
\usepackage{subfigure}
\usepackage{sagetex}

\usepackage{comment}

\usepackage{url} 

\numberwithin{equation}{section}

\usepackage{tikzsymbols}
\usepackage{tikzlings}

\usetikzlibrary{decorations.pathreplacing, angles, quotes}

\newcommand{\Z}{\mathbb{Z}}

\newcommand{\R}{\mathbb{R}}
\newcommand{\C}{\mathbb{C}}

\definecolor{lilac}{rgb}{0.67, 0.44, 7.96} 
\definecolor{ultramarine}{rgb}{0.17, 0.38, 7.96} 
\definecolor{iris}{rgb}{0.47, 0.08, 5.96} 
\definecolor{limegreen}{rgb}{0.5, 0.7, 0.3}

\usepackage{type1cm}        
\usepackage{makeidx}         
\usepackage{graphicx}        
\usepackage{multicol}        
\usepackage[bottom]{footmisc}

\usepackage{newtxtext}       %
\usepackage[varvw]{newtxmath}       

\makeindex             

\begin{document}

\title*{Symmetry groups of origami structures}
\author{Sara Chari and Andrew Quinn Macauley}
\institute{Sara Chari \at St. Mary's College of Maryland\\ 18952 E. Fisher Rd.\\ St. Mary's City, MD 20686, \email{slchari@smcm.edu}
\and Andrew Quinn Macauley \at Name, Address of Institute \email{quinnmacauley@gmail.com}}
%
%
\maketitle

\abstract*{Origami is the art of folding paper into various patterns without cutting or tearing the paper. By viewing the complex plane as an infinite piece of paper and the angles as folds, we iteratively compute and record all intersection points of allowed lines to construct mathematical origami sets. Additionally, we include the various lines to create a repeating pattern that can be viewed as a wallpaper group if the angle set contains 3 or fewer angles. There are 17 wallpaper groups up to isomorphism, so we determine which such groups can be constructed in this way, depending on the rotational and reflectional symmetries present in the given pattern. If the angle set contains more than three angles, the resulting pattern will be dense and hence no longer a wallpaper group. In this case, the classification of the symmetry group is done algorithmically.}

\section{Introduction}
Mathematical origami is an algebraic interpretation of traditional origami in which one folds a paper without cutting or tearing it. To obtain a mathematical origami construction, one may view the complex plane as an infinite piece of paper, and draw lines in the plane that correspond to ``folds". The construction begins with two points called ``seed points" (typically 0 and 1) and a set $U$ of allowed angles. For any two points $p,q \in \C$ and any two angles $\alpha, \beta \in U$, the intersection of the line through point $p$ at an angle of $\alpha$ with a line through point $q$ at an angle $\beta$ is denoted by $[\![p,q]\!]_{\alpha, \beta}$, as shown in Figure (\ref{intersection_viz}) below.
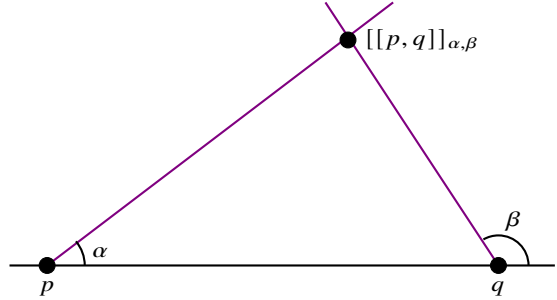
\begin{figure}[h!t]
\centering
\begin{tikzpicture}

\coordinate (A) at (0,0);
\coordinate (B) at (6,0);
\coordinate (P) at (0,-0.1);
\coordinate (Q) at (6,-0.1);
\coordinate (R) at (6.75,0);
\coordinate (Mid) at (4,3);
\coordinate (PreA) at (-0.5, 0);
\coordinate (LMid) at (3.7, 3.5);
\coordinate (RMid) at (4.6, 3.5);

\draw[black, thick] (PreA) -- (R);
\draw[violet, thick] (RMid) -- (A);
\draw[violet, thick] (LMid) -- (B);

\filldraw[black] (A) circle (3pt) node[below] {};
\filldraw[black] (B) circle (3pt) node[below] {};
\filldraw[black] (Mid) circle (3pt) node[right] {\phantom{s}$[[p,q]]_{\alpha,\beta}$};

\draw[black, thick] (A) ++(0.5,0) arc[start angle=0,end angle=40,radius=0.5] node[midway, right] {$\alpha$};
\draw[black, thick] (B) ++(0.4,0) arc[start angle=0,end angle=120,radius=0.4] node[above right] {\phantom{s} $\beta$};

\node at (P) [below] {$p$}; 
\node at (Q) [below] {$q$}; 

\end{tikzpicture}
\label{intersection_viz}
\caption{The intersection of $[\![p,q]\!]_{\alpha,\beta}$.}
\end{figure}

The two lines, denoted by $L_{p,\alpha}$ and $L_{q,\beta}$, are of the form $p+r(\cos(\alpha)+i\sin(\alpha))$ and $q+s(\cos(\beta)+i\sin(\beta))$ respectively, for $r,s \in \R$. The intersection of these two lines can be found by solving the equation
    
\begin{equation} \label{intformula}
p+r(\cos(\alpha)+i\sin(\alpha))=q+s(\cos(\beta)+i\sin(\beta))
\end{equation}
for $r$ and $s$ and substituting their values back into either side of the equation. The process of constructing an origami set is iterative and goes as follows. 

\begin{definition}
    Let $U$ be a set of angles. Define $$M_1(U) \colonequals \{[\![0,1]\!]_{\alpha,\beta}: \alpha, \beta \in U, \alpha \neq \beta\}.$$ For each integer $k >1$, define $$M_k(U)\colonequals\{[\![p,q]\!]_{\alpha,\beta}: p,q \in M_{k-1}(U), \alpha,\beta \in U\}.$$

    The \textbf{origami set} is defined as $$M(U) \colonequals \bigcup_k M_k(U).$$

    Then, the \textbf{origami structure} is defined as $$S(U) \colonequals  \bigcup_{\alpha \in U}  \bigcup_{p \in M(U)} L_{p,\alpha}.$$
   
\end{definition}
In other words, the origami set is the set of all points obtained via intersections of lines at allowed angles through points already obtained in this way, and the origami structure includes all lines at allowed angles through all points in the origami set. Throughout this paper, it is assumed that the angle set $U$ is finite so that there are a finite number of points obtained in each iteration. It is also assumed that $U$ contains the angle 0.


Depending on the angle set $U$, several known conditions will determine whether the resulting set $M(U)$ will be a lattice, a ring, dense, or some combination of these. The set of symmetries of an origami structure forms a group, and classifying such groups is the main focus of this paper. The study of symmetry patterns has deep ties to crystallography (in higher dimensions), chemistry (including molecular symmetry), and other natural objects. Combining the subjects of origami constructions and wallpaper groups can allow for the identification of various symmetry patterns, including fractals, which has applications in image processing, artificial intelligence/machine learning, and solar panel configuration.

Before proceeding to construct origami sets, we will begin with a brief overview of wallpaper groups and their properties. 
\begin{definition}
    A wallpaper group is the group of symmetries of an infinite repeating pattern in 2-dimensional space such that the set of images of a given point under the translation subgroup is a 2-dimensional lattice.
\end{definition}

There are 17 wallpaper groups up to isomorphism, and based on their symmetries, they are unique in pattern and characteristics. 
The types of symmetries include reflections across a given axis, rotations about a given point, glide reflections, translations, and any combination of these. Given an angle set $U$ and the corresponding origami set $M(U)$, we define $G(U)$ to be the wallpaper group corresponding to the given origami structure. It is often useful to consider the \emph{point group}, $P(U)$, which is the wallpaper group modulo the subgroup of translations.

When classifying each wallpaper group, the numbers 1,2,3,4, and 6 represent the highest rotational order within the subgroup of rotations modulo $2\pi$. Each wallpaper pattern is made up of a cell-based lattice structure, and the shape of the cells helps determine the symmetries of the wallpaper group. These cells can be one of five different types of polygons: parallelograms, rectangles, rhombi, squares, and hexagons \cite{Armstrong}. 

The letters $p$, $c$, $m$, and $g$ are used to represent different characteristics of the corresponding symmetries. The letter $p$ is used to classify groups with a \emph{primitive} lattice structure, where no cell within the lattice contains a lattice point; this occurs in the classification of all but two of the 17 different patterns. If $p$ is not used, the letter $c$ is used to denote a lattice with a \emph{centered} cell structure in which each cell includes a lattice point at its center; of the 17 wallpaper groups, two are classified with a \textit{c}. The letter $m$ is used to denote reflections, or \emph{mirrors}, within the lattice. Similarly, $g$ is used to denote \emph{glide-reflections}, which are the combination of a translation and a reflection or vice versa.  
\begin{figure}[h!t]
    \centering
    \includegraphics[width=4in]{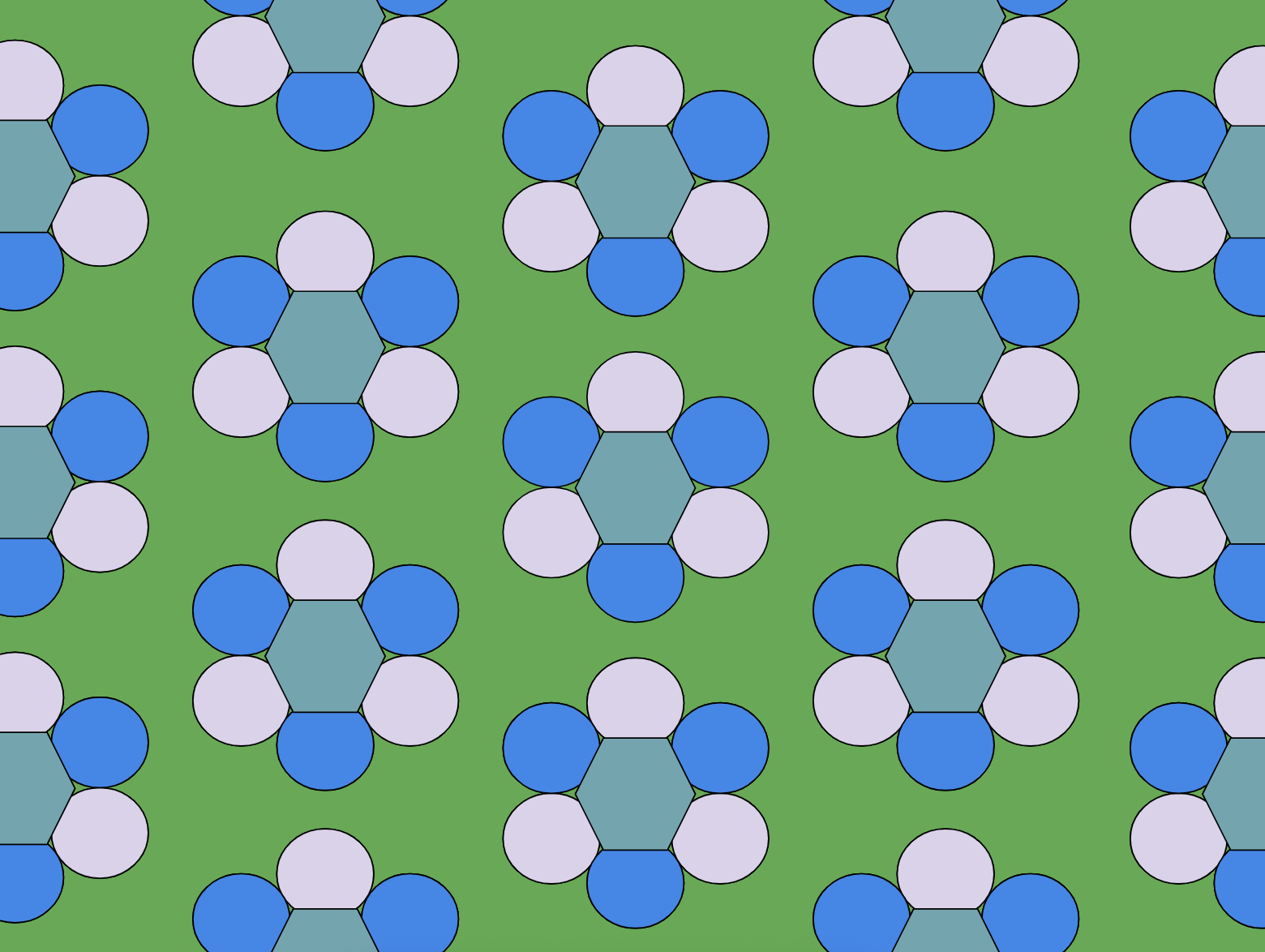}
    \caption{A wallpaper group characterized as p31m.}
    \label{fig:intro-3}
\end{figure}

When the angle set $U$ for an origami structure $S(U)$ contains exactly three angles, $S(U)$ is a wallpaper group and falls into one of 17 different groups up to isomorphism. However, of the 17 possible wallpaper groups, only 3 can be constructed via mathematical origami. 

In Section 2, we describe known results about origami constructions that will be useful for the main results. In Sections 3 and 4, we will describe the criteria for an origami structure to contain certain rotations and reflections, respectively. Section 5 includes the main results about the possible wallpaper groups formed in the case where the origami set is a lattice, and Section 6 describes the possible structures of the symmetry groups if $U$ contains more than 3 angles, in which case the structure is dense and is no longer a wallpaper group. Origami constructions are computational by nature, which allows for algorithms to output a set of points that generate an origami set. 

\section{Computing intersections and projections}

 We now discuss previously known results about origami constructions and the properties of the corresponding origami sets. The first factor is the number of allowed angles. The following theorem allows for a quick way to check if the set is dense.
 
 \begin{theorem} \cite{BBDLG}

Let $U$ be a set of angles where $|U|>3$. Then, $M(U)$ is dense in $\C$.    

  \end{theorem}
 

When an angle set $U$ contains exactly three angles, the origami set $M(U)$ is a lattice rather than dense, and the lattice structure can be computed using the following theorem.

\begin{theorem} \cite{Nedrenco}
 Given a set $U$ of three angles and two seed points 0 and 1, the origami set $M(U)$ is given by $\Z+\Z \tau,$ where $\tau=[\![0,1]\!]_{\alpha,\beta}$.
\end{theorem}

The value of $\tau$ can be calculated using the intersection formula given in Equation \ref{intformula}.
To calculate $r$ and $s$, Equation \ref{intformula} must be viewed as two equations by separating it into two parts: the real part and the imaginary part. Writing $p = p_r + ip_i$ and $q = q_r + iq_i$ where $p_r$ and $q_r$ are the real parts of $p$ and $q$ respectively, and $p_i$ and $q_i$ are the imaginary parts yields 
$$p_r + r \cdot \cos (\alpha) = q_r + s \cdot \cos(\beta) \; \; \textrm{and} \; \; p_i + r \cdot \sin (\alpha) = q_i + s \cdot \sin(\beta).$$
Then, rearranging the equations yields  
$$r \cdot \cos (\alpha) - s \cdot \cos(\beta) = q_r - p_r  \; \; \textrm{and} \; \; r \cdot \sin (\alpha) - s \cdot \sin(\beta) = q_i - p_i.$$
Thus, the following linear system is obtained
\begin{equation}\label{linear tau}
    \begin{pmatrix}
    \cos (\alpha) & - \cos(\beta) \\
    \sin (\alpha) & -\sin(\beta)
\end{pmatrix} 
\begin{pmatrix}
    r \\ s
\end{pmatrix} = 
\begin{pmatrix}
    q_r - p_r \\ q_i - p_i
\end{pmatrix},
\end{equation}
and can be solved for $r$ and $s$ to find $\tau=p+r(\cos(\alpha)+i\sin(\alpha))$. 

It is important to recognize how points and vectors in $\R^2$ can be represented on the complex plane and vice versa. There is a frequent switch between the two spaces throughout this paper as needed.

We will make use of the following definitions and theorem due to Barr and Roth \cite{BR} in the later sections of this paper.
\begin{definition} \label{projint}
    Given a set $U$ of angles, the set $S$ of \textbf{initial intersections} is defined as $$S=\{[\![0,1]\!]_{\alpha,\beta}: \alpha,\beta \in U\}$$ and the set $P$ of \textbf{projections} onto the real line is defined as $$P=\{[\![0,q]\!]_{0,\alpha}: q \in S, \alpha \in U\}.$$
\end{definition}
\begin{theorem} \cite{BR} \label{BRthm}
Let $U$ be a set of angles, let $S=\{s_1, s_2, \dots, s_m\}$ be the set of initial intersections, and let $P$ be the set of projections. Define $R$ to be the ring $R=\Z[P]$. Then, $M(U)$ is the $R$-module generated by $S$; i.e., $$M(U)=Rs_1+Rs_2+\cdots +Rs_m.$$    
\end{theorem}

The table below contains an algorithm that calculates the set of initial intersections using Equation (\ref{linear tau}) for a given list of possible angles and seed points by applying the linear algebra highlighted above. This algorithm was developed into a MATLAB script that was used to verify calculations. 

\begin{table}[h!t]
    \centering
    \begin{tabular}{|p{0.5cm} p{12cm}|}
    \hline
    \multicolumn{2}{|c|}{\textbf{Algorithm 1}} \\
    \hline
    1 & \underline{Input:} A vector $V$ of $m$ angles and seed points $p$ and $q$.\\
    2 & \underline{Output:} A set $S$ that contains all unique  intersection points using angles from $V$ and points $p$ and $q$. \\
    3 & $S = \emptyset$\\
    4 & \textbf{for} $k = 1, 2, \dots, m:$\\
    5 & \; $\big|$ \ \textbf{for} $j = k + 1, \dots, m$:\\
    6 & \; $\big|$ \; \ $\big|$ \; $\alpha = V(k), \; \beta = V(j)$\\
    7 & \; $\Bigg|$ \; \ $\Bigg|$ \; $V = \begin{pmatrix}
        \cos(\alpha) & -\cos(\beta) \\ \sin(\alpha) & -\sin(\beta)
    \end{pmatrix}$ \\
    8 & \; $\Bigg|$ \; \ $\Bigg|$ \; $a = \begin{pmatrix}
        \text{Re}(q) - \text{Re}(p) \\ \text{Im}(q) - \text{Im}(p)
    \end{pmatrix}$ \\
    9 & \; $\big|$ \; \ $\big|$ \; $b = V \backslash a$ \\
    10 & \; $\Bigg|$ \; \ $\Bigg|$ \; $u = \begin{pmatrix}
        \cos(\alpha) & 0 \\ \sin(\alpha) & 0
    \end{pmatrix} \cdot b + \begin{pmatrix}
        \text{Re}(p) \\ \text{Im}(p)
    \end{pmatrix}$ \\
    11 & \; $\big|$ \; \ $\big|$ \; $\tau  = \begin{pmatrix}
        1 & i
    \end{pmatrix} \cdot u$ \\
    12 & \; $\big|$ \; \ $\big|$ \; $S  = S \cup \tau$ \\
    13 & \; $\big|$ \ \textbf{end} \\
    14 & \textbf{end} \\
    15 & \textbf{Return} $S$ \\
    \hline
    \end{tabular}
    \label{}
\end{table}


Below is another algorithm that solves Equation (\ref{intformula}); however, in this case, $p = 0$ and $\alpha = 0$. The algorithm is solving the equation $r = q + s\cdot(\cos(\beta) + i\cdot\sin(\beta))$ for a list of possible values of angles $\beta$ in $U$ and points $q$ in $S$ in the same way as in Algorithm 1, but this time the output is the set of projections. This is used to compute the projection of points onto the real axis via the set of allowed angles. These projections are used for determining the generators for the base ring of an origami set. 
\pagebreak

\begin{table}[h!t]
    \centering
    \begin{tabular}{|p{0.5cm} p{12cm}|}
    \hline
    \multicolumn{2}{|c|}{\text{\textbf{Algorithm 2}}} \\
    \hline
    1 & \underline{Input:} A vector $S$ of $n$ initial intersections and a vector $V$ of $m$ angles.\\
    2 & \underline{Output:} A set $P$ that contains all unique real projections using initial intersections from $S$ and angles from $V$. \\
    3 & $V' = V - \{0\}$ \\
    4 & $l = |V'|$ \\
    5 & $P = \emptyset$ \\
    6 & \textbf{for} $k = 1, 2, \dots, l:$\\
    7 & \; $\big|$ \ $c = V'(k)$\\
    8 & \; $\big|$ \ \textbf{for} $j = 1, \dots, n$:\\
    9 & \; $\big|$ \; \ $\big|$ \; $\beta = S(j)$ \\
    10 & \; $\Bigg|$ \; \ $\Bigg|$ \; $T = \begin{pmatrix}
        1 & -\cos(\beta) \\ 0 & -\sin(\beta)
    \end{pmatrix}$ \\
    11 & \; $\Bigg|$ \; \ $\Bigg|$ \; $a = \begin{pmatrix}
        \text{Re}(c)\\ \text{Im}(c)
    \end{pmatrix}$ \\
    12 & \; $\big|$ \; \ $\big|$ \; $b = T \backslash a$ \\
    13 & \; $\big|$ \; \ $\big|$ \; $\tau  = \begin{pmatrix}
        1 & 0
    \end{pmatrix} \cdot b $\\
    14 & \; $\big|$ \; \ $\big|$ \; $P = P \cup \tau$\\
    15 & \; $\big|$ \ \textbf{end}\\
    16 & \textbf{end}\\
    17 & \textbf{Return} $P$ \\
    \hline
    \end{tabular}
    \label{tab:my_label}
\end{table}

These algorithms are used in tandem to compute $M(U)$, a process that was first outlined by Bahr and Roth \cite{BR}.
The MATLAB code for each algorithm can be found in a GitHub repository using the following link:
\url{https://github.com/quinnmacauley20/Wallpaper-Groups-in-Origami-Structures}. Algorithm 1 corresponds to the code found in \textit{tau.m} and Algorithm 2 corresponds to the code found in \textit{projection.m}.

Below is a code written using SageMath which also inputs and angle set $U$ and outputs the sets $S$ and $P$ described in Definition \ref{projint}
\begin{sageblock}
def proj(U):
    S=[]
    P=[]
    p=matrix(SR,[[0],[0]])
    q=matrix(SR,[[1],[0]])
    for alpha in U:
        for beta in U:
            if alpha<beta:
                M1=matrix(SR,[[cos(alpha),0],[sin(alpha),0]])
                M2=matrix(SR,[[cos(alpha),-cos(beta)],[sin(alpha),-sin(beta)]])
                rs=(M2.inverse())*(q-p)
                tau=(M1*rs)
                if tau not in S:
                    S.append(tau)
    print(S)
    for tau in S:
        for beta in U:
            if beta!=0:
                M1=matrix(SR,[1,0])
                M2=matrix(SR,[[1,-cos(beta)],[0,-sin(beta)]])
                rs=(M2.inverse())*(tau)
                r=M1*rs
                if r[0] not in P:
                     P.append(r[0])
    print(P)
\end{sageblock}

\section{Rotations}

We now explore the various possible symmetries of wallpaper groups constructed via mathematical origami. A major characteristic of a wallpaper group is its rotational symmetries, which will have orders 1, 2, 3, 4, or 6. 
A rotation of order 1 is the identity rotation, which is contained in all wallpaper groups. For a group to be an origami structure, the wallpaper group must contain a rotation of order 2, which corresponds to negation. To show the rotations algebraically, the rotation of a vector by $\theta$ radians corresponds to multiplication by the matrix
$$\begin{pmatrix}
    \cos(\theta) & -\sin(\theta) \\
    \sin(\theta) & \cos(\theta)
\end{pmatrix}.$$
For example, rotating the vector $\begin{pmatrix} 1 \\ 0 \end{pmatrix}$, or $ 1 + 0i$ on the complex plane, yields the vector $\begin{pmatrix} 0 \\ 1 \end{pmatrix}$, or $0 + 1i$ when  $\theta=\frac{\pi}{2}$. This operation is seen as
$$\begin{pmatrix}
    \cos(\frac{\pi}{2}) & -\sin(\frac{\pi}{2}) \\
    \sin(\frac{\pi}{2}) & \cos(\frac{\pi}{2})
\end{pmatrix}  \begin{pmatrix} 1 \\ 0 \end{pmatrix} = 
\begin{pmatrix}
    0 & -1 \\
    1 & 0
\end{pmatrix} \begin{pmatrix} 1 \\ 0 \end{pmatrix} = \begin{pmatrix} 0 \\ 1 \end{pmatrix} .$$

An angle set $U=\{0,\alpha, \beta\}$ containing three directions will generate an origami set $M(U)$, and an origami structure with symmetry group $G(U)$. Figure (\ref{abg}) below shows the triangles that are formed in this process, where $\gamma=\pi-\beta$ and $\rho=\beta-\alpha$. The angles $\alpha, \gamma$, and $\rho$ will from now on be referred to as the ``triangle angles" as they are the angles within each triangle formed as shown in the figure below. 



\begin{figure}[h!t] \label{abg}
\centering
\begin{tikzpicture}

\coordinate (A) at (0,0);
\coordinate (B) at (6,0);
\coordinate (C) at (3,3);
\coordinate (D) at (-1,0);
\coordinate (E) at (7,0);
\coordinate (F) at (-1,3);
\coordinate (G) at (7,3);
\coordinate (H) at (-1, -1);
\coordinate (I) at (4, 4);
\coordinate (J) at (5, -1);
\coordinate (K) at (7, 1);
\coordinate (L) at (1, -1);
\coordinate (M) at (-1, 1);
\coordinate (N) at (2, 4);
\coordinate (O) at (7, -1);
\coordinate (P) at (0, -1/4);
\coordinate (Q) at (6, -1/4);
\coordinate (R) at (2.15, 2.65);
\coordinate (S) at (5.1, 0.4);

\draw[ultramarine, thick] (D) -- (E);
\draw[ultramarine, thick] (F) -- (G);
\draw[limegreen, thick] (M) -- (L);
\draw[limegreen, thick] (N) -- (O);
\draw[violet, thick] (H) -- (I);
\draw[violet, thick] (J) -- (K);

\filldraw[black, thick] (A) circle (3pt) node[below] {};
\filldraw[black, thick] (B) circle (3pt) node[below] {};
\filldraw[black, thick] (C) circle (3pt) node[right] {};

\draw[black, thick] (A) ++(0.5,0) arc[start angle=0,end angle=45,radius=0.5] node[midway, right] {$\alpha$};
\draw[black, thick] (B) ++(0.4,0) arc[start angle=0,end angle=135,radius=0.4] node[above right] {\phantom{s} $\beta$};
\draw[black, thick] (R) ++(0.5,0) arc[start angle=225, end angle=315, radius=.5] node[below left] {\phantom{its} $\rho$ \phantom{i}};
\draw[black, thick] (S) ++(0.5,0) arc[start angle=135, end angle=180, radius=.6] node[above left] {$\gamma$};

\node at (P) [below] {0}; 
\node at (Q) [below] {1}; 

\end{tikzpicture}
\label{alpha_beta_gamma}
\caption{The triangle angles $\alpha$, $\rho$, and $\gamma$.}
\end{figure}
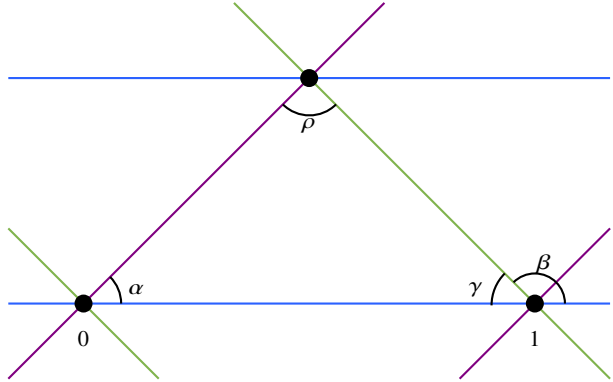

The following lemmas provide information about how the angles relate to each other based on the rotations present in the corresponding wallpaper group. 

\begin{lemma}\label{equi}
    Let $U=\{0, \alpha, \beta\}$ be a set of three angles with corresponding triangle angles $\alpha$, $\rho$, and $\gamma$. Then, $G(U)$ contains a rotation by $\theta \neq k\pi$ for any $k \in \Z$ if and only if $\alpha=\rho=\gamma=\frac{\pi}{3}$.
\end{lemma}

\begin{proof}
Let $U=\{0,\alpha,\beta\}$. First, suppose $G(U)$ contains a rotation by $\theta \neq k\pi$ for any $k \in \Z$, and suppose that $\theta$ is the smallest such angle. We may assume that $0<\alpha<\beta<\pi$ and that $0<\theta<\pi$. Modulo $2\pi$, the angle 0 can only be mapped to an angle in the set $\{0, \alpha, \beta, \pi, \pi+\alpha, \pi+\beta\}$. Since $\theta<\pi$, this means a rotation by $\theta$ must map the angle 0 to either $\alpha$ or $\beta$, so either $\theta=\alpha$ or $\theta=\beta$. Furthermore, since there is a rotation by an angle of $\pi$, there is also a rotation by an angle of $\pi-\theta$. If $\theta>\frac{\pi}{2}$ then $\pi-\theta\leq \frac{\pi}{2}$ which is not possible since $\theta$ was chosen to be the smallest angle that is not a multiple of $\pi$. Therefore, we may assume that $0<\theta\leq \frac{\pi}{2}$. 



If $\theta=\frac{\pi}{2}$, then either $\alpha=\frac{\pi}{2}$ or $\beta=\frac{\pi}{2}$. If $\alpha=\frac{\pi}{2}$. Then, $\beta>\frac{\pi}{2}$ and $\beta-\theta=\beta-\frac{\pi}{2}>0$ must also be a nonzero allowed angle that is less than $\pi$, since that is where the line through the origin at angle $\beta$ is mapped under the rotation by $-\theta$. Therefore, either $\beta-\frac{\pi}{2}=\beta$ which is not true, or $\beta-\frac{\pi}{2}=\alpha=\frac{\pi}{2}$ so $\beta=\pi$ which is also not the case. If $\beta=\frac{\pi}{2}$, then, $\alpha<\frac{\pi}{2}$, so $alpha$ is mapped to $\alpha+\theta=\alpha+\frac{\pi}{2}<\pi$ which must be another allowed angle less than $\pi$ so either $\alpha+\frac{\pi}{2}=\alpha$ which is not true, or $\alpha+\frac{\pi}{2}=\beta=\frac{\pi}{2}$ and $\alpha=0$ which is also not the case. Therefore, $\theta\neq \frac{\pi}{2}$ so $\theta<\frac{\pi}{2}$.

If $\theta<\frac{\pi}{2}$, then $G(U)$ must contain $\theta$, $2\theta$ and $\pi-\theta$, all of which are less than $\pi$. Note that $\theta<2\theta$ and $\theta<\pi-\theta$ so $\theta$ is the smallest of the three angles. Since $\alpha<\beta$ it must be the case that $\theta=\alpha$, which leaves $\beta=2\theta$ and $\beta=\pi-\theta$. Therefore, $2\theta=\pi-\theta$ and hence $\theta=\frac{\pi}{3}$ leaving $\alpha=\frac{\pi}{3}$ and $\beta=2\theta=\frac{2\pi}{3}$. Finally, this yields $\gamma=\pi-\beta=\frac{\pi}{3}$ and $\rho=\pi-\alpha-\gamma=\frac{\pi}{3}$.

   Conversely, suppose that $\alpha=\rho=\gamma=\frac{\pi}{3}$. Then, $\beta=\pi-\gamma=\frac{2\pi}{3}$ and so $U=\{0, \frac{\pi}{3}, \frac{2\pi}{3}\}$. There is an angle at every multiple of $\frac{\pi}{3}$ since we may add $\pi$ to any angle, so we now need to show that for any $p\in M(U)$, its rotation by the angle $\theta=\frac{\pi}{3}$ is also in $M(U)$. 
   The initial intersection is $\tau=[\![0,1]\!]_{\alpha,\beta}=[\![0,1]\!]_{\frac{\pi}{3}, \frac{2\pi}{3}}=\frac{1+\sqrt{3}i}{2}$. As shown by Nedrenco \cite{Nedrenco}, $M(U)=\Z+\Z\tau=\{m+n\tau: m,n \in \Z\}$. Note that $m+n\tau=(m+\frac{n}{2})+\left(\frac{n\sqrt{3}}{2}\right)i$ which is represented by the vector $\begin{pmatrix}
        m+\frac{n}{2}\\
        \frac{n\sqrt{3}}{2}
    \end{pmatrix}$.
   
Using the rotational matrix with $\theta=\frac{\pi}{3}$ yields
    $$
   \begin{pmatrix}
        \frac{1}{2} & -\frac{\sqrt{3}}{2} \\
        \frac{\sqrt{3}}{2} & \frac{1}{2}
    \end{pmatrix}
   \begin{pmatrix}
        m+\frac{n}{2}\\
        \frac{n\sqrt{3}}{2}
    \end{pmatrix}=\begin{pmatrix}
        \frac{1}{2}(m-n)\\
        \frac{\sqrt{3}}{2}(m+n)
    \end{pmatrix}.$$
    This corresponds to the point $-n+(m+n)\left(\frac{1+\sqrt{3}i}{2}\right)$ which is also in $M(U)$ since $-n$ and $m+n$ are integers. Therefore, there is a rotation of order $\frac{\pi}{3}\neq k\pi$ as desired.
\end{proof}

As the only origami structure that contains rotations by an angle smaller than $\pi$, the structure discussed above has a unique quality that will be valuable when determining the corresponding wallpaper group when the triangles formed are equilateral.






\section{Reflections}

The other major type of symmetry is reflections, so in this section, we study the properties of reflections in the context of origami constructions. 

To show the reflections algebraically, the reflection of a vector by $\theta$ radians corresponds to multiplication by the matrix
$$\begin{pmatrix}
    \cos(2\theta) & \sin(2\theta) \\
    \sin(2\theta) & -\cos(2\theta)
\end{pmatrix}.$$

For example, reflecting the vector $\begin{pmatrix} 1 \\ 0 \end{pmatrix}$, or $ 0 + i$ on the complex plane, across the angle $\theta$ yields the vector $\begin{pmatrix} \cos(2\theta) \\ \sin(2\theta) \end{pmatrix}$. This operation is seen as
$$\begin{pmatrix}
    \cos(2\theta) & \sin(2\theta) \\
    \sin(2\theta) & \cos(2\theta)
\end{pmatrix}  \begin{pmatrix} 1 \\ 0 \end{pmatrix} = 
\begin{pmatrix} \cos(2\theta) \\ \sin(2\theta) \end{pmatrix}.$$
Like rotational symmetries, the reflectional symmetries of a given origami wallpaper group are determined by the shape of the triangles comprising the structure. 

The following lemma shows where each angle is mapped under a reflection across a given axis.


\begin{lemma} \label{reflect}
    Let $\theta$ and $\eta$ be angles with $0\leq \theta, \eta <2\pi$. Then, when reflected across the line through the origin at angle $\theta$, the angle  $\eta$ is mapped to $2\theta-\eta$. 
\end{lemma}

\begin{proof}
    The line through the origin at angle $\eta$ is the set of points represented by all scalar multiples of the vector $\begin{pmatrix}
        \cos(\eta) \\ \sin(\eta)
    \end{pmatrix}$. When reflected across the line through the origin at angle $\theta$, the vector $\begin{pmatrix}
        \cos(\eta) \\ \sin(\eta)
    \end{pmatrix}$ is mapped to 

$$\begin{pmatrix}
\cos(2\theta) & \sin(2\theta) \\
\sin(2\theta) & -\cos(2\theta)
\end{pmatrix} \begin{pmatrix}
\cos(\eta) \\
\sin(\eta)
\end{pmatrix}=
\begin{pmatrix}
\cos(2\theta)\cos(\eta)+\sin(2\theta)\sin(\eta)\\
\sin(2\theta)\cos(\eta)-\cos(2\theta)\sin(\eta)
\end{pmatrix} = \begin{pmatrix}
\cos(2\theta-\eta)\\
\sin(2\theta-\eta)
\end{pmatrix}.$$
By linearity, any scalar multiple of  $\begin{pmatrix}
        \cos(\eta) \\ \sin(\eta)
    \end{pmatrix}$ will map to the corresponding scalar multiple of $\begin{pmatrix}
        \cos(2\theta-\eta) \\ \sin(2\theta-\eta)
    \end{pmatrix}$, yielding the entire line through the origin at angle $2\theta-\eta.$
\end{proof}

It turns out that if $G(U)$ contains a reflection across the angle $\theta$ then it also contains a reflection across the angle $\theta+\frac{\pi}{2}$ as demonstrated in the next lemma.

\begin{lemma} \label{reflect2} Let $U=\{0, \alpha, \beta\}$ be an angle set. If $G(U)$ contains a reflection across the angle $\theta$, then $G(U)$ must also contain a reflection across the angle $\theta=\frac{\pi}{2}$.
\end{lemma}

\begin{proof}
    Suppose that $G(U)$ contains reflections across the angle $\theta$. The  matrix for a reflection across the line $\theta$ is
$$\begin{pmatrix}
\cos(2\theta) & \sin(2\theta) \\
\sin(2\theta) & -\cos(2\theta)
\end{pmatrix} .$$

The  matrix for a reflection across the angle $\theta +\frac{\pi}{2}$ is
$$\begin{pmatrix}
\cos(2\theta+\pi) & \sin(2\theta+\pi) \\
\sin(2\theta+\pi) & -\cos(2\theta+\pi)
\end{pmatrix} =\begin{pmatrix}
-\cos(2\theta) & -\sin(2\theta) \\
-\sin(2\theta) & \cos(2\theta)
\end{pmatrix} =\begin{pmatrix}
\cos(2\theta) & \sin(2\theta) \\
\sin(2\theta) & -\cos(2\theta)
\end{pmatrix} \begin{pmatrix}
-1 & 0 \\
0 & -1
\end{pmatrix} .$$

Therefore, the reflection across the angle $\theta+\frac{\pi}{2}$ is the composition of a reflection across the angle $\theta$ and a rotation by $\pi$ (i.e., negation), both of which are symmetries in $G(U)$. It follows that the reflection across the angle $\theta+\frac{\pi}{2}$ is also in $G(U)$.

    Now let $\eta \in U$ be an arbitrary angle. Then, by Lemma \ref{reflect}, the angle $\eta$ is mapped to the angle $2\theta-\eta$ under the reflection by the angle $\theta$ by Lemma \ref{reflect}. Therefore, $2\theta-\eta \in U$. Again by Lemma \ref{reflect}, the angle $\eta$ is mapped to the angle $2\left(\theta+\frac{\pi}{2}\right)-\eta=2\theta+\pi-\eta=2\theta-\eta$. But, $2\theta-\eta \in M(U)$, so for any angle $\theta \in U$, its reflection aross the angle $\theta+\frac{\pi}{2}$ is also in $U$, so $G(U)$ contains reflections across the angle $\theta+\frac{\pi}{2}$.
\end{proof}

The next few lemmas relate the reflection axes of an origami structure to the angles within the triangles of the structure. As it turns out, the reflection axes present determine the equality of two angles within the triangles of the origami structure.

\begin{lemma} \label{alphagamma}
    Let $U = \{0,\alpha, \beta\} $ be a set of three angles with corresponding triangle angles $\alpha, \rho$, and $\gamma$. Then, $\alpha=\gamma$ if and only if $G(U)$ group contains reflections about the angles $0$ and $\frac{\pi}{2}.$ 
    
\end{lemma}

\begin{proof}
Suppose $\alpha = \gamma$, so $\beta=\pi-\alpha$ and $U=\{0, \alpha, \pi-\alpha\}.$ Then, using Equation \ref{intformula} and results from \cite{Nedrenco}, $M(U) = \Z + \Z \tau$ where $\tau = [\![0,1]\!]_{\alpha, \beta}=\frac{1}{2}(1+i\tan(\alpha))$.  The  matrix for a reflection across the line $\theta = 0$ is
$$\begin{pmatrix}
1 & 0 \\
0 & -1
\end{pmatrix}.$$
Using the reflection matrix and viewing 1 and $\tau$ as the vectors
$\begin{pmatrix} 1 \\ 0 \end{pmatrix}$ and
$\begin{pmatrix} \frac{1}{2} \\ \frac{1}{2}\tan(\alpha) \end{pmatrix}$, 
1 and $\tau$ are mapped to 1 and $1-\tau$, respectively. Therefore, any $m + n\tau \in M(U)$ will be reflected to $m + n - n\tau$ when $\theta = 0$, where $m, n \in \Z$, showing that $M(U)$ is closed under reflections across $\theta = 0$. By Lemma \ref{reflect}, The angle $0$ is mapped to $0-0=0$, the angle $\alpha$ is mapped to $-\alpha=\pi-\alpha$, and the angle $\pi-\alpha$ is mapped to $0-(\pi-\alpha)=\alpha$ under the reflection across the angle $0$. Therefore, $G(U)$ contains a reflection across the angle $\theta=0$. By Lemma \ref{reflect2}, $G(U)$ also contains a reflection across the angle $\frac{\pi}{2}$.


Conversely, suppose that $G(U)$ contains reflections about the angles $0$ and $\frac{\pi}{2}.$ Note that the origin has three lines through it at angles $0, \alpha$, and $\pi-\gamma$ where $\alpha,\gamma \neq 0$. When reflected across the line through the origin at angle $0$, the line through the origin at angle $\alpha$ is mapped to the line through the origin at angle $-\alpha$, which is the same as the line through the origin at angle $\pi-\alpha$. Therefore, either $\pi-\alpha=0$ or $\pi-\alpha=\pi-\gamma$. If $\pi-\alpha=0$ then $\alpha=\pi$ so $\gamma=\pi-\pi=0$ which is not the case since $\gamma \neq 0$. Therefore, $\pi-\alpha=\pi-\gamma$ so $\alpha=\gamma.$
\end{proof}

\begin{lemma} \label{alphabeta}
   Let $U = \{0,\alpha, \beta\} $ be a set of three angles with corresponding triangle angles $\alpha, \rho=\beta-\alpha$, and $\gamma=\pi-\beta$. Then, $\alpha=\rho$ if and only if $G(U)$ group contains reflections about the angles $\frac{\alpha + \rho}{2}$ and $\frac{\alpha + \rho + \pi}{2}$.
\end{lemma}

\begin{proof}
Suppose that $\alpha=\rho$, so that $\beta=2\alpha$ and $U=\{0, \alpha, 2\alpha\}.$ Then, using Equation \ref{intformula} and results from \cite{Nedrenco}, $M(U) = \Z + \Z \tau$ where $\tau = [\![0,1]\!]_{\alpha, \beta} = \cos(2\alpha)+i\sin(2\alpha).$
The matrix for a reflection across $\theta = \frac{\alpha + \rho}{2} = \alpha$ is given by
$$\begin{pmatrix}
\cos(2\alpha) & \sin(2\alpha) \\
\sin(2\alpha) & -\cos(2\alpha)
\end{pmatrix}.$$
Using the reflection matrix and viewing 1 and $\tau$ as vectors 
$\begin{pmatrix} 1 \\ 0 \end{pmatrix}$ and
$\begin{pmatrix} \cos(2\alpha) \\ \sin(2\alpha) \end{pmatrix}$, 
$1$ and $\tau$ are mapped to $\tau$ and $1$, respectively. Therefore, any $m + n\tau \in M(U)$ is mapped to $n+m\tau$ when $\theta = \frac{\alpha+\rho}{2}=\alpha$, showing $M(U)$ is closed under reflections across $\frac{\alpha+\rho}{2}$. By Lemma \ref{reflect}, the angle 0 is mapped to $2\alpha$, the angle $2\alpha$ is mapped to $2\alpha-2\alpha=0$, and the angle $\alpha$ is mapped to $2\alpha-\alpha=\alpha$ under the reflection across the angle $\alpha$. Therefore, $G(U)$ contains a reflection across the angle $\theta=\alpha=\frac{\alpha+\rho}{2}$.
By Lemma \ref{reflect2},  $G(U)$ also contains a reflection across the angle $\frac{\alpha+\rho+\pi}{2}$.

Conversely, suppose that $G(U)$ contains reflections about the angles $\frac{\alpha+\rho}{2}$ and $\frac{\alpha+\rho+\pi}{2}.$ Note that the origin has three lines through it at angles $0, \alpha$, and $\pi-\gamma=\alpha+\rho$ where $\alpha, \rho, \gamma \neq 0$. When reflected across the line through the origin at angle $\frac{\alpha+\rho}{2}$, the line through the origin at angle $\alpha$ is mapped to the line through the origin at angle $2\left(\frac{\alpha+\rho}{2}\right)-\alpha=\rho$. Therefore, either $\rho=0$, $\rho=\alpha$, or $\rho=\alpha+\rho$ in which case $\alpha=0$. Since $\alpha\neq 0$ and $\rho \neq 0$ it must be true that $\rho=\alpha$.
\end{proof}

\begin{lemma} \label{betagamma}
    Let $U = \{0,\alpha, \beta\} $ be a set of three angles with corresponding triangle angles $\alpha, \rho=\beta-\alpha$, and $\gamma=\pi-\beta$. Then, $\rho=\gamma$ if and only if $G(U)$ contains reflections about the angles $\frac{\alpha}{2}$ and $\frac{\alpha+\pi}{2}$.
\end{lemma}

\begin{proof}
    Suppose that $\rho = \gamma$, so $\alpha=\pi-2\rho$, $\beta=\pi-\rho$, and $U = \{0, \alpha, \pi -\rho\}.$ Then, using Equation \ref{intformula} and results from \cite{Nedrenco}, $M(U) = \Z + \Z \tau$ where $\tau =[\![0,1]\!]_{\alpha, \beta} = \cos(\pi-2\rho)+i\sin(\pi-2\rho)=\cos(\alpha)+i\sin(\alpha).$
    
   The matrix for a reflection across $\theta = \frac{\alpha}{2}$ is given by
$$\begin{pmatrix}
\cos(\alpha) & \sin(\alpha) \\
\sin(\alpha) & -\cos(\alpha)
\end{pmatrix}.$$
Using the reflection matrix and viewing 1 and $\tau$ as vectors 
$\begin{pmatrix} 1 \\ 0 \end{pmatrix}$ and
$\begin{pmatrix} \cos(\alpha) \\ \sin(\alpha) \end{pmatrix}$, 
$1$ and $\tau$ are mapped to $\tau$ and $1$, respectively. Therefore, any $m + n\tau \in M(U)$ is mapped to $n+m\tau$ when $\theta = \frac{\alpha+\rho}{2}=\alpha$, showing $M(U)$ is closed under reflections across $\frac{\alpha}{2}$. By Lemma \ref{reflect}, the angle $0$ is mapped to the angle $2\left(\frac{\alpha}{2}\right)-0=\alpha$, the angle $\alpha$ is mapped to the angle $\alpha-\alpha=0$, and the angle $\pi-\rho$ is mapped to the angle $\alpha-(\pi-\rho)=\alpha-(\alpha+\gamma)=-\gamma=\pi-\rho$. Therefore, $G(U)$ contains a reflection across the angle $\theta=\frac{\alpha}{2}$. By Lemma \ref{reflect2}, $G(U)$ also contains a reflection across the angle $\frac{\alpha+\pi}{2}$.

Conversely, suppose that $G(U)$ contains reflections about the angles $\frac{\alpha}{2}$ and $\frac{\alpha + \pi}{2}.$ Note that the origin has three lines through it at angles $0, \alpha$, and $\rho$ where $\alpha,\gamma \neq 0$.
When reflected across the line through the origin at angle $\frac{\alpha}{2}$, the line through the origin at angle $\alpha+\rho$ is mapped to the line through the origin at angle $-\rho$ which is the same as the line through the origin at angle $\pi-\rho$. Therefore, either $\pi-\rho=0$, $\pi-\rho=\alpha$, or $\pi-\rho=\pi-\gamma$. If $\pi-\rho=\alpha$ then $\alpha+\rho=\pi$, but $\alpha+\rho=\pi-\gamma$, meaning that $\gamma=0$. Since $\rho<\pi$ and $\gamma \neq 0$, it must be the case that $\pi-\rho=\pi-\gamma$, or $\rho=\gamma$.
\end{proof}

Combining Lemmas \ref{alphagamma}, \ref{alphabeta}, and \ref{betagamma}, it follows that if $G(U)$ contains reflections about an angle $\theta$ with $0\leq \theta<\pi$, then at least one of the following holds:
 
 \begin{enumerate}
     \item $\alpha=\gamma$ and $\theta=0$ or $\theta=\frac{\pi}{2}$
     \item $\alpha=\rho$ and $\theta=\frac{\alpha+\rho}{2}$ or $\theta=\frac{\alpha+\rho+\pi}{2}$
     \item $\rho = \gamma$ and $\theta = \frac{\alpha}{2}$ or $\theta = \frac{\alpha + \pi}{2}$
 \end{enumerate}
In other words, the wallpaper pattern contains isosceles triangles.

\section{Main results for three angles}

In this section, we combine all of the lemmas from previous sections regarding which reflections and rotations are present in the origami wallpaper group generated by a given set of three angles. As previously discussed, the origami structures generated are made up of triangles such that the different lattice types can be subdivided into individual triangles. This unique characteristic will come into play when determining the different wallpaper groups generated by the triangles.

\begin{theorem} \label{mainthm}
Let $U = \{0,\alpha, \beta\} $ be a set of three angles with corresponding triangle angles $\alpha, \rho=\beta-\alpha$, and $\gamma=\pi-\beta$. If the triangles within $S(U)$ are scalene, isosceles, or equilateral, then the corresponding wallpaper groups generated by the triangles are p2, cmm, and p6m, respectively.
 \end{theorem}

 \begin{proof}
     If the triangles within the origami structure are scalene, then the only rotations in $G(U)$ are multiples of $\pi$ radians by Lemma \ref{equi}. Moreover, if $\alpha\neq\rho \neq \gamma$, then by Lemmas \ref{alphagamma}, \ref{alphabeta}, and \ref{betagamma}, $G(U)$ contains no reflections. Therefore, $G(U)$ corresponds to the wallpaper group p2.

     If the triangles within the origami structure are isosceles but not equilateral, then the only rotations in $G(U)$ are multiples of $\pi$ radians by Lemma \ref{equi}. Moreover, by Lemmas \ref{alphagamma}, \ref{alphabeta}, and \ref{betagamma}, $G(U)$ contains exactly two perpendicular reflection axes in all cases. Since there is also a glide reflection axis parallel to the bases of the isosceles triangles, $G(U)$ corresponds to the wallpaper group cmm.

     If the triangles within the origami structure are equilateral, then by Lemma \ref{equi}, $G(U)$ contains rotations by $\frac{\pi}{3}$ which have order 6. Moreover, $\alpha=\rho= \gamma=\frac{\pi}{3}$ and so by Lemmas \ref{alphagamma}, \ref{alphabeta}, and \ref{betagamma}, $G(U)$ contains the 6 reflection axes $0, \frac{\pi}{6}, \frac{\pi}{3}, \frac{\pi}{2}, \frac{2\pi}{3}$, and $\frac{5\pi}{6}$. Therefore, $G(U)$ corresponds to the wallpaper group p6m.

 \end{proof}




The following examples show three different wallpaper structures corresponding to the three wallpaper groups that can be obtained via origami construction.

\begin{example} \label{p2}  $U=\{0, \arcsin\left(\frac{2\sqrt{5}}{5}\right), \frac{\pi}{2}\}$

    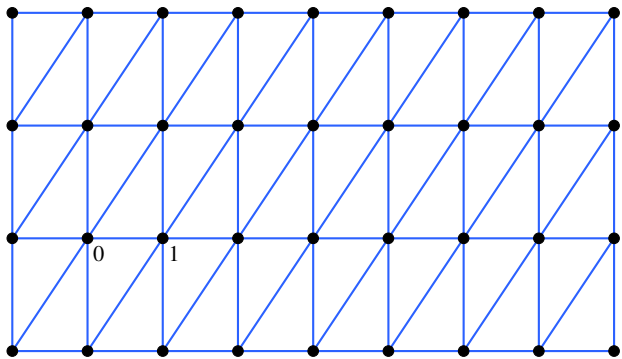
\begin{figure}[h!t]
\centering
\begin{tikzpicture}

\coordinate (A) at (-1,0);
\coordinate (B) at (0,0);
\coordinate (C) at (1,0);
\coordinate (D) at (2,0);
\coordinate (E) at (3,0);
\coordinate (F) at (4,0);
\coordinate (G) at (5,0);
\coordinate (H) at (6,0);
\coordinate (I) at (7,0);
\coordinate (J) at (-1,3/4*2);
\coordinate (K) at (0,3/4*2);
\coordinate (L) at (1,3/4*2);
\coordinate (M) at (2,3/4*2);
\coordinate (N) at (3,3/4*2);
\coordinate (O) at (4,3/4*2);
\coordinate (P) at (5,3/4*2);
\coordinate (Q) at (6,3/4*2);
\coordinate (R) at (7,3/4*2);
\coordinate (S) at (-1,3/4*4);
\coordinate (T) at (0,3/4*4);
\coordinate (U) at (1,3/4*4);
\coordinate (V) at (2,3/4*4);
\coordinate (W) at (3,3/4*4);
\coordinate (X) at (4,3/4*4);
\coordinate (Y) at (5,3/4*4);
\coordinate (Z) at (6,3/4*4);
\coordinate (AA) at (7,3/4*4);
\coordinate (AB) at (-1,3/4*-2);
\coordinate (AC) at (0,3/4*-2);
\coordinate (AD) at (1,3/4*-2);
\coordinate (AE) at (2,3/4*-2);
\coordinate (AF) at (3,3/4*-2);
\coordinate (AG) at (4,3/4*-2);
\coordinate (AH) at (5,3/4*-2);
\coordinate (AI) at (6,3/4*-2);
\coordinate (AJ) at (7,3/4*-2);

\coordinate (AK) at (0, 16/5);
\coordinate (AL) at (0, -8/5);

\draw[ultramarine, thick] (AB)--(V);
\draw[ultramarine, thick] (AC)--(W);
\draw[ultramarine, thick] (AD)--(X);
\draw[ultramarine, thick] (AE)--(Y);
\draw[ultramarine, thick] (AF)--(Z);
\draw[ultramarine, thick] (AG)--(AA);
\draw[ultramarine, thick] (A)--(U);
\draw[ultramarine, thick] (J)--(T);
\draw[ultramarine, thick] (AH)--(R);
\draw[ultramarine, thick] (AI)--(I);
\draw[ultramarine, thick] (AB)--(AJ);
\draw[ultramarine, thick] (S)--(AA);
\draw[ultramarine, thick] (J)--(R);
\draw[ultramarine, thick] (A)--(I);
\draw[ultramarine, thick] (AB)--(S);
\draw[ultramarine, thick] (AC)--(T);
\draw[ultramarine, thick] (AD)--(U);
\draw[ultramarine, thick] (AE)--(V);
\draw[ultramarine, thick] (AF)--(W);
\draw[ultramarine, thick] (AG)--(X);
\draw[ultramarine, thick] (AH)--(Y);
\draw[ultramarine, thick] (AI)--(Z);
\draw[ultramarine, thick] (AJ)--(AA);


\filldraw[black] (A) circle (2pt) node[below] {};
\filldraw[black] (B) circle (2pt) node[below] {\phantom{00}0};
\filldraw[black] (C) circle (2pt) node[below] {\phantom{00}1};
\filldraw[black] (D) circle (2pt) node[below] {};
\filldraw[black] (E) circle (2pt) node[below] {};
\filldraw[black] (F) circle (2pt) node[below] {};
\filldraw[black] (G) circle (2pt) node[below] {};
\filldraw[black] (H) circle (2pt) node[below] {};
\filldraw[black] (I) circle (2pt) node[below] {};
\filldraw[black] (J) circle (2pt) node[below] {};
\filldraw[black] (K) circle (2pt) node[below] {};
\filldraw[black] (L) circle (2pt) node[below] {};
\filldraw[black] (M) circle (2pt) node[below] {};
\filldraw[black] (N) circle (2pt) node[below] {};
\filldraw[black] (O) circle (2pt) node[below] {};
\filldraw[black] (P) circle (2pt) node[below] {};
\filldraw[black] (Q) circle (2pt) node[below] {};
\filldraw[black] (R) circle (2pt) node[below] {};
\filldraw[black] (S) circle (2pt) node[below] {};
\filldraw[black] (T) circle (2pt) node[below] {};
\filldraw[black] (U) circle (2pt) node[below] {};
\filldraw[black] (V) circle (2pt) node[below] {};
\filldraw[black] (W) circle (2pt) node[below] {};
\filldraw[black] (X) circle (2pt) node[below] {};
\filldraw[black] (Y) circle (2pt) node[below] {};
\filldraw[black] (Z) circle (2pt) node[below] {};
\filldraw[black] (AA) circle (2pt) node[below] {};
\filldraw[black] (AB) circle (2pt) node[below] {};
\filldraw[black] (AC) circle (2pt) node[below] {};
\filldraw[black] (AD) circle (2pt) node[below] {};
\filldraw[black] (AE) circle (2pt) node[below] {};
\filldraw[black] (AF) circle (2pt) node[below] {};
\filldraw[black] (AG) circle (2pt) node[below] {};
\filldraw[black] (AH) circle (2pt) node[below] {};
\filldraw[black] (AI) circle (2pt) node[below] {};
\filldraw[black] (AJ) circle (2pt) node[below] {};




\end{tikzpicture}
\label{scalene_ex}
\caption{The scalene triangles of $S(U)$ correspond to the wallpaper group $p2$.}
\end{figure}

\end{example}

\begin{example} \label{cmm} $U=\{0, \frac{\pi}{4}, \frac{\pi}{2}\}$

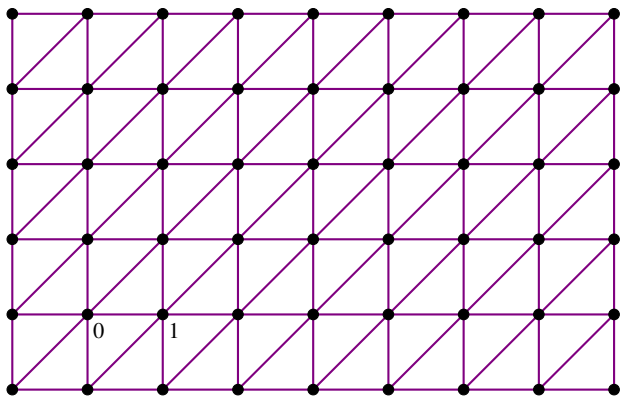
\begin{figure}[h!t]
\centering
\begin{tikzpicture}

\coordinate (A) at (-1,0);
\coordinate (B) at (0,0);
\coordinate (C) at (1,0);
\coordinate (D) at (2,0);
\coordinate (E) at (3,0);
\coordinate (F) at (4,0);
\coordinate (G) at (5,0);
\coordinate (H) at (6,0);
\coordinate (I) at (7,0);
\coordinate (J) at (-1,1);
\coordinate (K) at (0,1);
\coordinate (L) at (1,1);
\coordinate (M) at (2,1);
\coordinate (N) at (3,1);
\coordinate (O) at (4,1);
\coordinate (P) at (5,1);
\coordinate (Q) at (6,1);
\coordinate (R) at (7,1);
\coordinate (S) at (-1,2);
\coordinate (T) at (0,2);
\coordinate (U) at (1,2);
\coordinate (V) at (2,2);
\coordinate (W) at (3,2);
\coordinate (X) at (4,2);
\coordinate (Y) at (5,2);
\coordinate (Z) at (6,2);
\coordinate (AA) at (7,2);
\coordinate (AB) at (-1,3);
\coordinate (AC) at (0,3);
\coordinate (AD) at (1,3);
\coordinate (AE) at (2,3);
\coordinate (AF) at (3,3);
\coordinate (AG) at (4,3);
\coordinate (AH) at (5,3);
\coordinate (AI) at (6,3);
\coordinate (AJ) at (7,3);
\coordinate (AK) at (-1,4);
\coordinate (AL) at (0,4);
\coordinate (AM) at (1,4);
\coordinate (AN) at (2,4);
\coordinate (AO) at (3,4);
\coordinate (AP) at (4,4);
\coordinate (AQ) at (5,4);
\coordinate (AR) at (6,4);
\coordinate (AS) at (7,4);
\coordinate (AT) at (-1,-1);
\coordinate (AU) at (0,-1);
\coordinate (AV) at (1,-1);
\coordinate (AW) at (2,-1);
\coordinate (AX) at (3,-1);
\coordinate (AY) at (4,-1);
\coordinate (AZ) at (5,-1);
\coordinate (BA) at (6,-1);
\coordinate (BB) at (7,-1);

\draw[violet, thick] (AB)--(AL);
\draw[violet, thick] (S)--(AM);
\draw[violet, thick] (J)--(AN);
\draw[violet, thick] (A)--(AO);
\draw[violet, thick] (AT)--(AP);
\draw[violet, thick] (AU)--(AQ);
\draw[violet, thick] (AV)--(AR);
\draw[violet, thick] (AW)--(AS);
\draw[violet, thick] (AX)--(AJ);
\draw[violet, thick] (AY)--(AA);
\draw[violet, thick] (AZ)--(R);
\draw[violet, thick] (BA)--(I);

\draw[violet, thick] (AK)--(AS);
\draw[violet, thick] (AB)--(AJ);
\draw[violet, thick] (S)--(AA);
\draw[violet, thick] (J)--(R);
\draw[violet, thick] (A)--(I);
\draw[violet, thick] (AT)--(BB);

\draw[violet, thick] (AT)--(AK);
\draw[violet, thick] (AU)--(AL);
\draw[violet, thick] (AV)--(AM);
\draw[violet, thick] (AW)--(AN);
\draw[violet, thick] (AX)--(AO);
\draw[violet, thick] (AY)--(AP);
\draw[violet, thick] (AZ)--(AQ);
\draw[violet, thick] (BA)--(AR);
\draw[violet, thick] (BB)--(AS);


\filldraw[black] (A) circle (2pt) node[below] {};
\filldraw[black] (B) circle (2pt) node[below] {\phantom{00}0};
\filldraw[black] (C) circle (2pt) node[below] {\phantom{00}1};
\filldraw[black] (D) circle (2pt) node[below] {};
\filldraw[black] (E) circle (2pt) node[below] {};
\filldraw[black] (F) circle (2pt) node[below] {};
\filldraw[black] (G) circle (2pt) node[below] {};
\filldraw[black] (H) circle (2pt) node[below] {};
\filldraw[black] (I) circle (2pt) node[below] {};
\filldraw[black] (J) circle (2pt) node[below] {};
\filldraw[black] (K) circle (2pt) node[below] {};
\filldraw[black] (L) circle (2pt) node[below] {};
\filldraw[black] (M) circle (2pt) node[below] {};
\filldraw[black] (N) circle (2pt) node[below] {};
\filldraw[black] (O) circle (2pt) node[below] {};
\filldraw[black] (P) circle (2pt) node[below] {};
\filldraw[black] (Q) circle (2pt) node[below] {};
\filldraw[black] (R) circle (2pt) node[below] {};
\filldraw[black] (S) circle (2pt) node[below] {};
\filldraw[black] (T) circle (2pt) node[below] {};
\filldraw[black] (U) circle (2pt) node[below] {};
\filldraw[black] (V) circle (2pt) node[below] {};
\filldraw[black] (W) circle (2pt) node[below] {};
\filldraw[black] (X) circle (2pt) node[below] {};
\filldraw[black] (Y) circle (2pt) node[below] {};
\filldraw[black] (Z) circle (2pt) node[below] {};
\filldraw[black] (AA) circle (2pt) node[below] {};
\filldraw[black] (AB) circle (2pt) node[below] {};
\filldraw[black] (AC) circle (2pt) node[below] {};
\filldraw[black] (AD) circle (2pt) node[below] {};
\filldraw[black] (AE) circle (2pt) node[below] {};
\filldraw[black] (AF) circle (2pt) node[below] {};
\filldraw[black] (AG) circle (2pt) node[below] {};
\filldraw[black] (AH) circle (2pt) node[below] {};
\filldraw[black] (AI) circle (2pt) node[below] {};
\filldraw[black] (AJ) circle (2pt) node[below] {};
\filldraw[black] (AK) circle (2pt) node[below] {};
\filldraw[black] (AL) circle (2pt) node[below] {};
\filldraw[black] (AM) circle (2pt) node[below] {};
\filldraw[black] (AN) circle (2pt) node[below] {};
\filldraw[black] (AO) circle (2pt) node[below] {};
\filldraw[black] (AP) circle (2pt) node[below] {};
\filldraw[black] (AQ) circle (2pt) node[below] {};
\filldraw[black] (AR) circle (2pt) node[below] {};
\filldraw[black] (AS) circle (2pt) node[below] {};
\filldraw[black] (AT) circle (2pt) node[below] {};
\filldraw[black] (AU) circle (2pt) node[below] {};
\filldraw[black] (AV) circle (2pt) node[below] {};
\filldraw[black] (AW) circle (2pt) node[below] {};
\filldraw[black] (AX) circle (2pt) node[below] {};
\filldraw[black] (AY) circle (2pt) node[below] {};
\filldraw[black] (AZ) circle (2pt) node[below] {};
\filldraw[black] (BA) circle (2pt) node[below] {};
\filldraw[black] (BB) circle (2pt) node[below] {};




\end{tikzpicture}
\label{isoc_ex}
\caption{The isosceles triangles of $S(U)$ correspond to the wallpaper group $cmm$.}
\end{figure}

\end{example}

\begin{example} \label{p6m} $U=\{0, \frac{\pi}{3}, \frac{2\pi}{3}\}$

\begin{figure}[h!t]
\centering
\begin{tikzpicture}

\coordinate (AZ) at (-1,2*1.73205080757);
\coordinate (BA) at (-1,1.73205080757);
\coordinate (BB) at (-1/2,3*1.73205080757/2);
\coordinate (BC) at (0,2*1.73205080757);
\coordinate (BD) at (-1,0);
\coordinate (BE) at (-1/2,1.73205080757/2);
\coordinate (BF) at (0,1.73205080757);
\coordinate (BG) at (1/2,3*1.73205080757/2);
\coordinate (BH) at (1,2*1.73205080757);
\coordinate (B) at (-1/2,-1.73205080757/2);
\coordinate (C) at (0,0);
\coordinate (D) at (1/2,1.73205080757/2);
\coordinate (E) at (1,1.73205080757);
\coordinate (F) at (3/2,3*1.73205080757/2);
\coordinate (G) at (2,2*1.73205080757);
\coordinate (I) at (1/2,-1.73205080757/2);
\coordinate (J) at (1,0);
\coordinate (K) at (3/2,1.73205080757/2);
\coordinate (L) at (2,1.73205080757);
\coordinate (M) at (5/2,3*1.73205080757/2);
\coordinate (N) at (3,2*1.73205080757);
\coordinate (P) at (3/2,-1.73205080757/2);
\coordinate (Q) at (2,0);
\coordinate (R) at (5/2,1.73205080757/2);
\coordinate (S) at (3,1.73205080757);
\coordinate (T) at (7/2,3*1.73205080757/2);
\coordinate (U) at (4,2*1.73205080757);
\coordinate (W) at (5/2,-1.73205080757/2);
\coordinate (X) at (3,0);
\coordinate (Y) at (7/2,1.73205080757/2);
\coordinate (Z) at (4,1.73205080757);
\coordinate (AA) at (9/2,3*1.73205080757/2);
\coordinate (AB) at (5,2*1.73205080757);
\coordinate (AD) at (7/2,-1.73205080757/2);
\coordinate (AE) at (4,0);
\coordinate (AF) at (9/2,1.73205080757/2);
\coordinate (AG) at (5,1.73205080757);
\coordinate (AH) at (11/2,3*1.73205080757/2);
\coordinate (AI) at (6,2*1.73205080757);
\coordinate (AK) at (9/2,-1.73205080757/2);
\coordinate (AL) at (5,0);
\coordinate (AM) at (11/2,1.73205080757/2);
\coordinate (AN) at (6,1.73205080757);
\coordinate (AO) at (13/2,3*1.73205080757/2);
\coordinate (AP) at (7,2*1.73205080757);
\coordinate (AR) at (11/2,-1.73205080757/2);
\coordinate (AS) at (6,0);
\coordinate (AT) at (13/2,1.73205080757/2);
\coordinate (AU) at (7,1.73205080757);
\coordinate (AW) at (13/2,-1.73205080757/2);
\coordinate (AX) at (7,0);
\coordinate (BI) at (7,3*1.73205080757/2);
\coordinate (BJ) at (-1,3*1.73205080757/2);
\coordinate (BK) at (7,1.73205080757/2);
\coordinate (BL) at (-1,1.73205080757/2);
\coordinate (BM) at (7,-1.73205080757/2);
\coordinate (BN) at (-1,-1.73205080757/2);

\draw[limegreen, thick] (BA)--(BC);
\draw[limegreen, thick] (BD)--(BH);
\draw[limegreen, thick] (B)--(G);
\draw[limegreen, thick] (I)--(N);
\draw[limegreen, thick] (P)--(U);
\draw[limegreen, thick] (W)--(AB);
\draw[limegreen, thick] (AD)--(AI);
\draw[limegreen, thick] (AK)--(AP);
\draw[limegreen, thick] (AR)--(AU);
\draw[limegreen, thick] (AW)--(AX);
\draw[limegreen, thick] (BA)--(I);
\draw[limegreen, thick] (BD)--(B);
\draw[limegreen, thick] (AZ)--(P);
\draw[limegreen, thick] (BC)--(W);
\draw[limegreen, thick] (BH)--(AD);
\draw[limegreen, thick] (G)--(AK);
\draw[limegreen, thick] (N)--(AR);
\draw[limegreen, thick] (U)--(AW);
\draw[limegreen, thick] (AB)--(AX);
\draw[limegreen, thick] (AI)--(AU);
\draw[limegreen, thick] (BD)--(AX);
\draw[limegreen, thick] (AZ)--(AP);
\draw[limegreen, thick] (BA)--(AU);
\draw[limegreen, thick] (BM)--(BN);
\draw[limegreen, thick] (BK)--(BL);
\draw[limegreen, thick] (BI)--(BJ);


\filldraw[black] (B) circle (2pt) node[below] {};
\filldraw[black] (C) circle (2pt) node[below] {\phantom{0000}0};
\filldraw[black] (D) circle (2pt) node[below] {};
\filldraw[black] (E) circle (2pt) node[below] {};
\filldraw[black] (F) circle (2pt) node[below] {};
\filldraw[black] (G) circle (2pt) node[below] {};
\filldraw[black] (I) circle (2pt) node[below] {};
\filldraw[black] (J) circle (2pt) node[below] {\phantom{0000}1};
\filldraw[black] (K) circle (2pt) node[below] {};
\filldraw[black] (L) circle (2pt) node[below] {};
\filldraw[black] (M) circle (2pt) node[below] {};
\filldraw[black] (N) circle (2pt) node[below] {};
\filldraw[black] (P) circle (2pt) node[below] {};
\filldraw[black] (Q) circle (2pt) node[below] {};
\filldraw[black] (R) circle (2pt) node[below] {};
\filldraw[black] (S) circle (2pt) node[below] {};
\filldraw[black] (T) circle (2pt) node[below] {};
\filldraw[black] (U) circle (2pt) node[below] {};
\filldraw[black] (W) circle (2pt) node[below] {};
\filldraw[black] (X) circle (2pt) node[below] {};
\filldraw[black] (Y) circle (2pt) node[below] {};
\filldraw[black] (Z) circle (2pt) node[below] {};
\filldraw[black] (AA) circle (2pt) node[below] {};
\filldraw[black] (AB) circle (2pt) node[below] {};
\filldraw[black] (AD) circle (2pt) node[below] {};
\filldraw[black] (AE) circle (2pt) node[below] {};
\filldraw[black] (AF) circle (2pt) node[below] {};
\filldraw[black] (AG) circle (2pt) node[below] {};
\filldraw[black] (AH) circle (2pt) node[below] {};
\filldraw[black] (AI) circle (2pt) node[below] {};
\filldraw[black] (AK) circle (2pt) node[below] {};
\filldraw[black] (AL) circle (2pt) node[below] {};
\filldraw[black] (AM) circle (2pt) node[below] {};
\filldraw[black] (AN) circle (2pt) node[below] {};
\filldraw[black] (AO) circle (2pt) node[below] {};
\filldraw[black] (AP) circle (2pt) node[below] {};
\filldraw[black] (AR) circle (2pt) node[below] {};
\filldraw[black] (AS) circle (2pt) node[below] {};
\filldraw[black] (AT) circle (2pt) node[below] {};
\filldraw[black] (AU) circle (2pt) node[below] {};
\filldraw[black] (AW) circle (2pt) node[below] {};
\filldraw[black] (AX) circle (2pt) node[below] {};
\filldraw[black] (AZ) circle (2pt) node[below] {};
\filldraw[black] (BA) circle (2pt) node[below] {};
\filldraw[black] (BB) circle (2pt) node[below] {};
\filldraw[black] (BC) circle (2pt) node[below] {};
\filldraw[black] (BD) circle (2pt) node[below] {};
\filldraw[black] (BE) circle (2pt) node[below] {};
\filldraw[black] (BF) circle (2pt) node[below] {};
\filldraw[black] (BG) circle (2pt) node[below] {};
\filldraw[black] (BH) circle (2pt) node[below] {};




\end{tikzpicture}
\label{eq_ex}
\caption{The equilateral triangles of $S(U)$ correspond to the wallpaper group $p6m$.}
\end{figure}
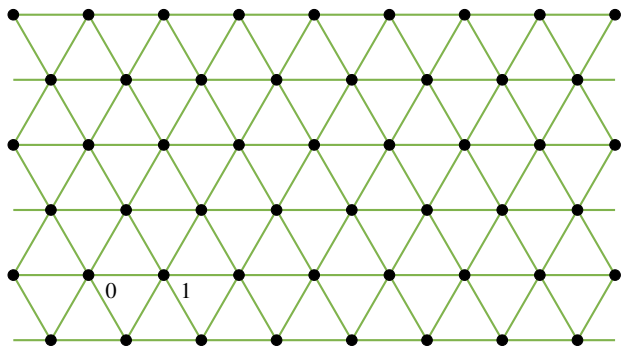

\end{example}

The following code was written in SageMath, and given an input set of three angles, it outputs the corresponding wallpaper group: either $p2$, $cmm$, or $p6m$.

\begin{sageblock}
def M(U):
    if len(U)>3:
        return(`error');
    if 0 not in U:
        return(`error');
    else: 
        U_ordered=[]; #orders the angles
        while len(U)>0:
            U_ordered.append(min(U));
            U.remove(min(U))
        alpha=U_ordered[1]
        gamma=pi-U_ordered[2]
        rho=U_ordered[2]-U_ordered[1]
        triangle_angles=[alpha,gamma,rho]
        a=0
        for i in range(0,2):
            for j in range(i+1,3):
                if triangle_angles[i]==triangle_angles[j]:
                    a=a+1
        if a==0:
            return(`p2')
        else:
            if a==1:
                return(`cmm')
            else:
                if a==3:
                    return(`p6m')
\end{sageblock}

\section{Obtaining other point groups}

Thus far, the three wallpaper groups obtained from origami constructions are those corresponding to origami sets that are lattices (equivalently, constructed from exactly three angles). If the angle set contains more than three angles, then the translation subgroup of the symmetry group is not a lattice and so the symmetry group can no longer be classified as a wallpaper group. In this case, we instead focus on just the corresponding point group of reflections and rotations, which will be denoted by $P(U)$ throughout this section. The group $P(U)$ contains a given reflection and rotation if and only if both $U$ and $M(U)$ are closed under the corresponding reflection or rotation. In this section, we determine the point groups obtained via origami constructions when the angle set contains more than three angles. We begin with criteria to obtain reflections and rotations.

\begin{proposition} \label{rot-ref}
Let $U$ be a set of angles, where $U$ is considered modulo $\pi$. Then,

\begin{enumerate}
\item  The group $P(U)$ contains rotations by an angle of $\theta$ if and only if for every $\alpha \in U$, we also have $\alpha+\theta \in U$ and $\cos(\theta)+i\sin(\theta)\in M(U)$.

\item The group $P(U)$ contains reflections across the angle $\theta$ if and only if for every $\alpha \in U$, we also have $2\theta-\alpha \in U$ and $\cos(2\theta)+i\sin(2\theta) \in M(U)$.
\end{enumerate}
\end{proposition}

\begin{proof}
To prove (1), first suppose that $P(U)$ contains rotations by the angle $\theta$. Under a rotation by $\theta$, an angle $\alpha$ is mapped to $\alpha + \theta$, and the seed point 1 is mapped to $\cos(\theta) + i \sin (\theta)$. Therefore, $\alpha+\theta\in U$ and $\cos(\theta)+i\sin(\theta) \in M(U)$.

Conversely, suppose for every $\alpha \in U$, we also have $\alpha + \theta \in U$ and $\cos(\theta) + i\sin(\theta) \in M(U)$. Therefore, the angle set is closed under rotations by $\theta$. Given any point $p \in M(U)$, the image of $p$ under a rotation by $\theta$ is obtained by the same steps used to construct $p$ but replacing every angle $\alpha$ by $\alpha + \theta$ and replacing the seed point 1 by $\cos(\theta) + i\sin(\theta)$. Therefore, $M(U)$ is also closed under rotations by $\theta$ and (1) holds.

To show (2), suppose first that $P(U)$ contains reflections across the angle $\theta$. Under a reflection across $\theta$, an angle $\alpha$ is mapped to $ 2\theta - \alpha$ by Lemma \ref{reflect}, and the seed point 1 is mapped to $\cos(2\theta) + i \sin (2\theta)$. Therefore, $2\theta - \alpha \in U$ and $\cos(2\theta)+i\sin(2\theta) \in M(U)$.

Conversely, suppose that for every $\alpha \in U$, $2\theta-\alpha \in U$ and $\cos(2\theta)+i\sin(2\theta)\in M(U)$. A reflection across an angle $\theta$ will map an angle $\alpha$ to $2\theta - \alpha$, so $U$ is closed under reflections across $\theta$. Given any point $p \in M(U)$, the image of $p$ under a reflection across $\theta$ is obtained by the same steps used to construct $p$ but replacing every angle $\alpha$ by $2\theta - \alpha$ and replacing the seed point 1 by $\cos(2\theta) + i\sin(2\theta)$. Therefore, $M(U)$ is also closed under reflections across $\theta$ and (2) holds. 

\end{proof}

Now that the criteria for rotations and reflections have been identified, they can be applied to determine the group structure of $P(U)$.

 \begin{theorem} \label{mainthm2}

Let $U$ be a set of angles, where $U$ is considered modulo $\pi$. Then,

 \begin{enumerate}

 \item If the highest rotation order is $n > 2$, then $P(U)$ is either cyclic of order $n$ if there are no reflections, or isomorphic to the dihedral group of a regular polygon with $n$ vertices if there are reflections.

 \item If the highest rotation order is $n=2$, then $P(U)$ is either cyclic of order 2 if there are no reflections, or isomorphic $\Z/2\Z \times \Z/2\Z$ if there are reflections. 

 \end{enumerate}
 \end{theorem}

\begin{proof}
    
   Let $n$ be the highest rotation order in $P(U)$. Combining a rotation by an angle $\theta_1$ followed by a reflection across an angle $\theta_2$ yields another reflection across the angle $\theta_2+\frac{\theta_1}{2}$. Similarly, combining two reflections yields a rotation by twice the difference in the reflection angles. This means there must either be no reflection axes or there are reflection axes that differ by exactly half of the smallest rotation angle. 
   
   In the former case, the symmetry group $P(U)\cong \Z/n\Z$ is cyclic of order $n$ as it contains only rotations. In the latter case if $n >2$, these are exactly the same symmetries as the dihedral group of a regular $n$-gon, proving (1). If $n=2$, there are four symmetries: one rotation, two reflections, and the identity, all of which have order dividing 2. Therefore, $P(U) \cong \Z/2\Z \times \Z/2\Z$, showing (2). 
\end{proof}

It should be noted that the dihedral group can be represented by generators (a rotation and a reflection) and relations. As an exercise to the reader, one could describe the group $P(U)$ in this same way and by obtaining the same relations, this would provide a more rigorous argument for the dihedral group case in the proof above.

We now consider the specific case where $U=\{0, \frac{\pi}{n}, \frac{2\pi}{n}, \dots, \frac{(n-1)\pi}{n}\}$.

\begin{theorem} \label{nodd}
    If $n$ is odd and $U=\{0,\frac{\pi}{n}, \frac{2\pi}{n}, \dots, \frac{(n-1)\pi}{n}\}$, then $P(U)\cong D_{2n}$, the dihedral group of a regular $2n$-gon.
\end{theorem}

\begin{proof} 
Given the angle set, the only possible rotations can be by multiples of $\frac{\pi}{n}$, and the only reflection axes can be angles that differ by $\frac{\pi}{2n}$, so it suffices to show that $P(U)$ contains a rotation by an angle of $\frac{\pi}{n}$ and a reflection across the axis given by the angle $\frac{\pi}{2n}$. 
Note that $[\![0,1]\!]_{\frac{\pi}{n}, \left(\frac{n+1}{2}\right)\left(\frac{\pi}{n}\right)}=\cos(\frac{\pi}{n})+i\sin(\frac{\pi}{n})$ by applying the intersection formula and various trigonometric identities. By Proposition \ref{rot-ref}, $P(U)$ contains a rotation by an angle of $\frac{\pi}{n}$.

Since $\cos(\frac{\pi}{n})+i\sin(\frac{\pi}{n}) \in M(U)$, $P(U)$ contains a reflection across $\frac{\pi}{2n}$ by Proposition \ref{rot-ref}.
 
\end{proof}

\begin{example}  
    Let $U = \{0, \frac{\pi}{5}, \frac{2\pi}{5}, \frac{3\pi}{5}, \frac{4\pi}{5}\}$. Then, by Proposition \ref{rot-ref} using $n=5$, $P(U)$ should be isomorphic to $D_{10}$. A rotation by $\frac{\pi}{5}$ maps $1$ to $\cos(\frac{\pi}{5}) + i\sin(\frac{\pi}{5})$ as given by both the intersection formula,
    $[\![0,1]\!]_{\frac{\pi}{5}, \frac{3\pi}{5}}$, and the application of the rotation matrix to $1$. This same process can be applied another $9$ times over until the identity of $1$ is reached, showing that $P(U)$ has a rotational order of $10$. Similarly, a reflection can be made on the starting point over an axis of $\frac{\pi}{10}$ so that $1$ is mapped to $\cos(\frac{\pi}{5}) + i\sin(\frac{\pi}{5})$. Again, this is given by both the intersection formula as well as the reflection matrix applied to $1$. As the dihedral group of the regular decagon $D_{10}$ is generated by a rotation of order 10 and a reflection, it is seen that $P(U)$ is isomorphic to $D_{10}$. 

    \begin{figure}[h!t]
    \centering
    \includegraphics[scale = .425]{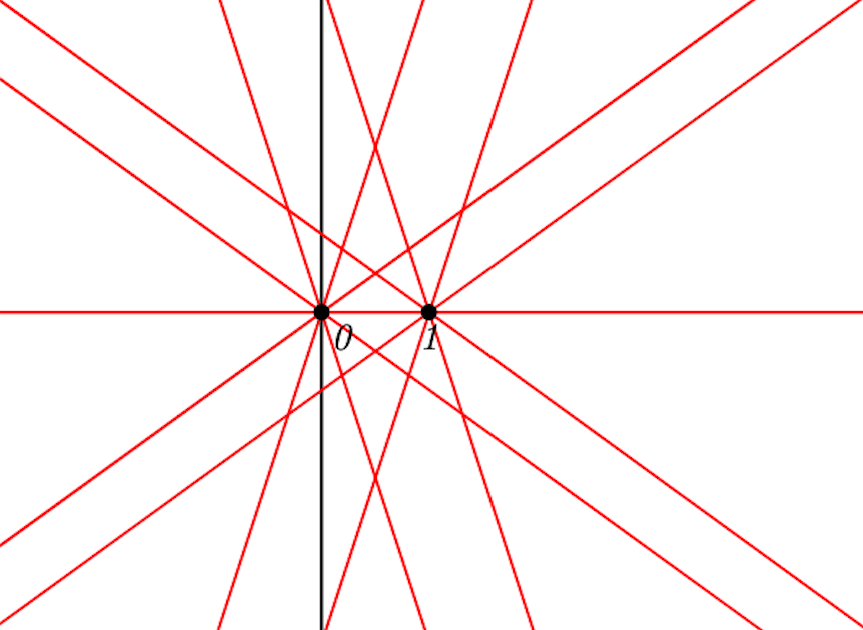}
    \includegraphics[scale = .7]{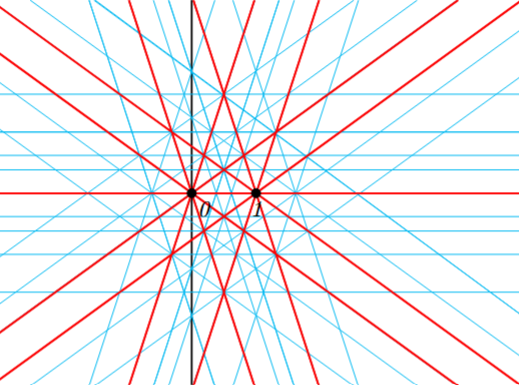}
    \caption{The images above show the first (left) and second (right) iterations of the intersection graph made up of the angles in $U=\{0, \frac{\pi}{5}, \frac{2\pi}{5}, \frac{3\pi}{5}, \frac{4\pi}{5}\}$. The red lines are obtained in the first iteration and the cyan lines are obtained in the second iteration.}
    \label{fig:enter-label}
\end{figure}

\end{example}

When $n$ is odd, ${n+1}$ is even and $\frac{n+1}{2}$ is an integer, which allows for rotations of order $2n$ in the previous theorem. When $n$ is even, there are no rotations of order $2n$ (i.e., a rotation by the angle $\frac{\pi}{n}$). If there were, then the image of the point 1 under the rotation would be $\cos(\frac{\pi}{n})+i\sin(\frac{\pi}{n})$. In the next series of lemmas, we will study the structure of $M(U)$ to see why this cannot be the case. The first lemma involves Chebyshev polynomials of the second kind which are defined as follows.

\begin{definition}

The $k$th \textbf{Chebyshev polynomial of the second kind}, denoted $U_k$ is defined so that $U_k(\cos(\theta))=\frac{\sin( (k+1)\theta)}{\sin(\theta)}$. 
\end{definition}

\begin{theorem} \cite{Pawelec}. The polynomials $U_k(x)$ follow the recursive formula $U_k(x)=2xU_{k-1}(x)-U_{k-2}(x)$, where $U_0(x)=1$ and $U_1(x)=2x$. 
\end{theorem}

\begin{lemma} \label{cheby}
For all $k \geq 0$, $U_k(x)$ will have degree $k$ and consist of either only even or only odd powers of $x$. Moreover, $\frac{U_{k-2}(x)}{U_k(x)}$ is a rational function with integer coefficients consisting of only even powers of $x$.
\end{lemma}

\begin{proof}
We will proceed by induction. When $k = 0$, $U_0(x) = x^0 = 1$, so $U_0$ has a degree 0. When $k = 1$, $U_1(x) = 2x$, showing $U_1$ has degree 1. 

Now, suppose $U_k$ and $U_{k - 1}$ each contain terms of only even or odd powers and have degrees of $k$ and $k - 1$, respectively. Note that $U_{k + 1}$ is defined as
$$U_{k+1} (x) = 2x \cdot U_k (x) - U_{k-1} (x).$$

In the case where $k$ is even, then both $k - 1$ and $k + 1$ are odd. Both $2x \cdot U_k (x)$ and $U_{k-1}(x)$ have terms of only odd powers of $x$ because $2x$ will alter $U_k (x)$ to have a degree of $k + 1$ which is odd, and $U_{k-1}(x)$ already has terms of only odd powers. Additionally, the difference between the two terms does not change the parity of the powers of the terms; so, $U_{k+1}$ has degree $k + 1$, an odd number such that the degree of each term in $U_{k+1}$ will be an odd power of $x$. 

In the case where $k$ is odd, the same argument holds, by reversing all parities.


Therefore, by induction, for all $k \geq 0$, $U_k(x)$ with degree $k$ will only have terms of $x$ with even or odd powers depending if $k$ is even or odd.

\end{proof}


\begin{lemma} \label{R}
Let $U=\{0, \frac{\pi}{n}, \frac{2\pi}{n}, \dots, \frac{(n-1)\pi}{n}\}$. Then, the set $P$ of projections contains only rational functions of $\cos^2(\frac{\pi}{n})$.
\end{lemma}

\begin{proof}
First, we show that all projections of $[\![0,1]\!]_{\frac{\pi}{n}, \frac{2\pi}{n}}$ onto the real axis, via angles in $U$ are rational functions of $\cos^2(\frac{\pi}{n}).$ Define $x_k$ to be the projection of the point $[\![0,1]\!]_{\frac{\pi}{n}, \frac{2\pi}{n}}$ onto the real axis, via the angle $\frac{(k+1)\pi}{n}$, as shown in Figure \ref{projections_tikz}. 
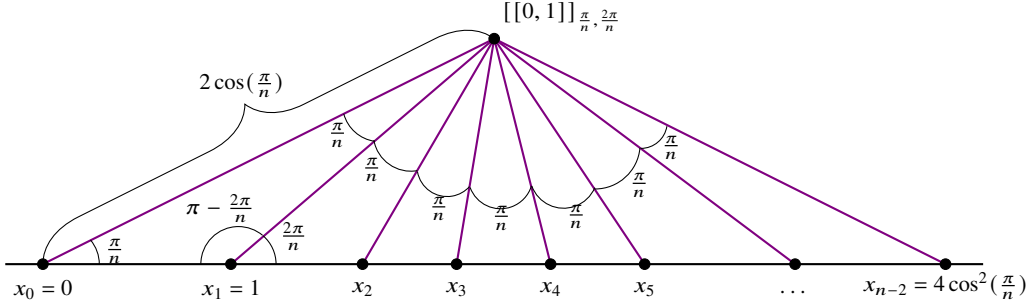
\begin{figure}[h!t]
    \centering
    \begin{tikzpicture}

\coordinate (A) at (0,0);
\coordinate (B) at (2.5,0);
\coordinate (C) at (4.25,0);
\coordinate (D) at (5.5,0);
\coordinate (E) at (6.75,0);
\coordinate (F) at (8,0);
\coordinate (G) at (10,0);
\coordinate (H) at (12,0);

\coordinate (Mid) at (6,3);
\coordinate (PreA) at (-0.5, 0);
\coordinate (PostH) at (12.5,0);
\coordinate (key) at (10,2);

\draw[black, thick] (PreA) -- (PostH);
\draw[violet, thick] (A) -- (Mid) -- (H);
\draw[violet, thick] (B) -- (Mid) -- (G);
\draw[violet, thick] (C) -- (Mid) -- (F);
\draw[violet, thick] (D) -- (Mid) -- (E);

\filldraw[black] (A) circle (2pt) node[below = 3] {$x_0 = 0$};
\filldraw[black] (B) circle (2pt) node[below = 3] {$x_1 = 1$};
\filldraw[black] (C) circle (2pt) node[below = 4] {$x_2$};
\filldraw[black] (D) circle (2pt) node[below = 4] {$x_3$};
\filldraw[black] (E) circle (2pt) node[below = 4] {$x_4$};
\filldraw[black] (F) circle (2pt) node[below = 4] {$x_5$};
\filldraw[black] (G) circle (2pt) node[below = 7pt] {\dots };
\filldraw[black] (H) circle (2pt) node[below] {$x_{n-2} = 4\cos^2(\frac{\pi}{n})$};
\filldraw[black] (Mid) circle (2pt) node[above right] {$[[0,1]]_{\frac{\pi}{n},\frac{2\pi}{n}}$};

\draw[black] (A) ++(0.75,0) arc[start angle=0,end angle=40,radius=0.5] node[midway, right] {$\frac{\pi}{n}$};
\draw[black] (B) ++(-0.4,0) arc[start angle=180,end angle=50,radius=0.5] node[above left, yshift = 2] {$\pi - \frac{2\pi}{n}$};
\draw[black] (B) ++(0.6,0) arc[start angle=0,end angle=49,radius=0.5] node[right] {\phantom{s}$\frac{2\pi}{n}$};

\draw[black] (Mid) ++(-2,-1) arc[start angle=195,end angle=262,radius=.5] node[midway, left, yshift = -1 pt] {$\frac{\pi}{n}$};
\draw[black] (4.4,1.65) arc[start angle=190,end angle=280,radius=0.5] node[midway, left, yshift = -2 pt] {$\frac{\pi}{n}$};
\draw[black] (4.98,1.23) arc[start angle=190,end angle=318,radius=0.4] node[midway, yshift = -5 pt] {$\frac{\pi}{n}$};
\draw[black] (5.68,1.02) arc[start angle=200,end angle=343,radius=0.43] node[midway, yshift = -3.5 pt] {$\frac{\pi}{n}$};
\draw[black] (6.5,1.04) arc[start angle=200,end angle=337,radius=0.45] node[xshift = -7 pt, yshift = -11 pt] {$\frac{\pi}{n}$};
\draw[black] (7.34,1) arc[start angle=270,end angle=355,radius=0.6] node[midway, right, yshift = -3 pt, xshift = -1.5] {$\frac{\pi}{n}$};
\draw[black] (7.95,1.54) arc[start angle=270,end angle=351,radius=0.35] node[midway, xshift = 5pt, yshift = -2 pt] {$\frac{\pi}{n}$};


\draw[decorate, decoration={brace, amplitude=20pt}] (A) -- (Mid) node[midway, yshift=25pt, xshift = -10pt] {$2\cos(\frac{\pi}{n})$};

\end{tikzpicture} 
    \caption{Projections of $[\![0,1]\!]_{\frac{\pi}{n}, \frac{2\pi}{n}}$ onto the real axis}
    \label{projections_tikz}
\end{figure}
Using the fact that $\sin(\pi-\frac{\pi}{n})=\sin(\frac{\pi}{n})$ along with the law of sines, we have $\frac{\sin(\frac{k\pi}{n})}{x_k}=\frac{\sin(\frac{(k+1)\pi}{n})}{2\cos(\frac{\pi}{n})}$, which yields $x_k=\frac{2\sin(\frac{k\pi}{n})\cos(\frac{\pi}{n})}{\sin(\frac{(k+1)\pi}{n})}$. This can be rewritten in terms of Chebyshev polynomials as $\frac{2\cos(\frac{\pi}{n})U_{k-1}(\cos(\frac{\pi}{n}))}{U_k(x)}=\frac{U_k(\cos(\frac{\pi}{n}))+U_{k-2}(\cos(\frac{\pi}{n}))}{U_k(\cos({\frac{\pi}{n}))}}=1+\frac{U_{k-2}(\cos(\frac{\pi}{n}))}{U_k(\cos({\frac{\pi}{n}))}}$, which was shown to be a rational function of $\cos^2(\frac{\pi}{n})$ in Lemma \ref{cheby}. Finally, define $x_{k}'$ to be the projection of $[\![0,1]\!]_{\frac{\pi(i+1)}{n}, \frac{(j+1)\pi}{n}}$ onto the real axis via the angle $\frac{(k+1)\pi}{n}$, for any integers $i,j \in \{0, \dots, n-2\}$, as shown in Figure \ref{projections_tikz3}. 
\begin{figure}[h!t]
    \centering
    \begin{tikzpicture}

\coordinate (A) at (0,0);
\coordinate (B) at (2.5,0);
\coordinate (C) at (4.25,0);
\coordinate (D) at (5.5,0);
\coordinate (E) at (6.75,0);
\coordinate (F) at (8,0);
\coordinate (G) at (10,0);
\coordinate (H) at (12,0);

\coordinate (Mid) at (6,3);
\coordinate (PreA) at (-0.5, 0);
\coordinate (PostH) at (12.5,0);
\coordinate (key) at (10,2);

\draw[black, thick] (PreA) -- (PostH);
\draw[violet, thick] (A) -- (Mid) -- (H);
\draw[violet, thick] (B) -- (Mid) -- (G);
\draw[violet, thick] (C) -- (Mid) -- (F);
\draw[violet, thick] (D) -- (Mid) -- (E);

\filldraw[black] (A) circle (2pt) node[below = 2.5] {$x_0'$};
\filldraw[black] (B) circle (2pt) node[below = 1] {$x_1'$};
\filldraw[black] (C) circle (2pt) node[below = 7] {\dots};
\filldraw[black] (D) circle (2pt) node[below = 1] {$x_i' = 0$};
\filldraw[black] (E) circle (2pt) node[below = 7] {\dots};
\filldraw[black] (F) circle (2pt) node[below = 1] {$x_j' = 1$};
\filldraw[black] (G) circle (2pt) node[below = 7pt] {\dots };
\filldraw[black] (H) circle (2pt) node[below = 1] {$x_{n-2}'$};
\filldraw[black] (Mid) circle (2pt) node[above right] {$[[0,1]]_{\frac{\pi}{n}(i+1),\frac{\pi}{n}(j+1)}$};

\draw[black] (A) ++(0.75,0) arc[start angle=0,end angle=40,radius=0.5] node[midway, right] {$\frac{\pi}{n}$};
\draw[black] (B) ++(-0.4,0) arc[start angle=180,end angle=50,radius=0.5] node[above left, yshift = 2] {$\pi - \frac{2\pi}{n}$};
\draw[black] (B) ++(0.6,0) arc[start angle=0,end angle=49,radius=0.5] node[right] {\phantom{s}$\frac{2\pi}{n}$};

\draw[black] (Mid) ++(-2,-1) arc[start angle=195,end angle=262,radius=.5] node[midway, left, yshift = -1 pt] {$\frac{\pi}{n}$};
\draw[black] (4.4,1.65) arc[start angle=190,end angle=280,radius=0.5] node[midway, left, yshift = -2 pt] {$\frac{\pi}{n}$};
\draw[black] (4.98,1.23) arc[start angle=190,end angle=318,radius=0.4] node[midway, yshift = -5 pt] {$\frac{\pi}{n}$};
\draw[black] (5.68,1.02) arc[start angle=200,end angle=343,radius=0.43] node[midway, yshift = -3.5 pt] {$\frac{\pi}{n}$};
\draw[black] (6.5,1.04) arc[start angle=200,end angle=337,radius=0.45] node[xshift = -7 pt, yshift = -11 pt] {$\frac{\pi}{n}$};
\draw[black] (7.34,1) arc[start angle=270,end angle=355,radius=0.6] node[midway, right, yshift = -3.2 pt, xshift = -1.5] {$\frac{\pi}{n}$};
\draw[black] (7.95,1.54) arc[start angle=270,end angle=351,radius=0.35] node[midway, xshift = 5pt, yshift = -2 pt] {$\frac{\pi}{n}$};



\end{tikzpicture} 
    \caption{Projections of $[\![0,1]\!]_{\frac{\pi(i+1)}{n}, \frac{(j+1)\pi}{n}}$ onto the real axis}
    \label{projections_tikz3}
\end{figure}
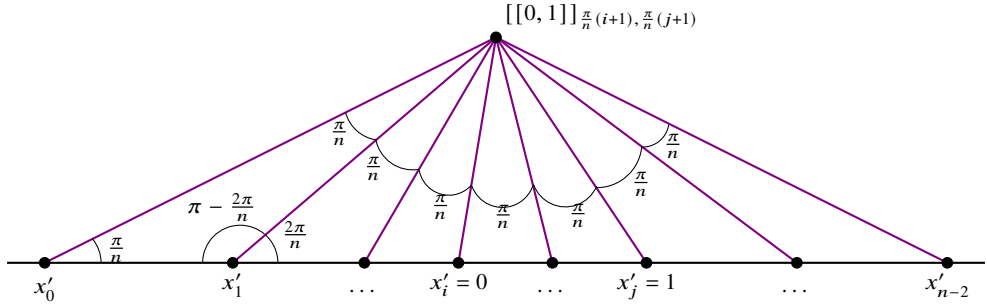

Note that the triangles in Figures \ref{projections_tikz} and \ref{projections_tikz3} are similar so $\frac{x_k'-x_i'}{x_k-x_i}=\frac{x_j'-x_i'}{x_j-x_i}$. Substituting $x_i'=0$ and $x_j'-x_i'=1$ and solving yields $x_k'=\frac{x_k-x_i}{x_j-x_i}$, which is still a rational function of $\cos^2(\frac{\pi}{n})$.
    
\end{proof}



\begin{theorem} \label{M(U)}
Let $U=\{0, \frac{\pi}{n}, \frac{2\pi}{n} \dots, \frac{(n-1)\pi}{n}\}$. Then, $M(U)=R+R(\frac{1}{2}+\frac{1}{2}\tan(\frac{\pi}{n})i)$ for some ring $R \subset \R$ where $R$ contains only rational functions of $\cos^2(\frac{\pi}{n})$ with integer coefficients.
\end{theorem}

\begin{proof}
Let $P$ be the set of projections of all initial intersections onto the real axis via angles in $U$, and let $P^{-1}$ be the set of all multiplicative inverses of elements of $P$. Define $R=\Z[P \cup P^{-1}]$. Then using results from \cite{BR}, $M(U)\cap \R=R$. By Lemma \ref{R}, $P$ contains only rational functions of $\cos^2(\frac{\pi}{n})$, and hence so does $R$. Since $R \subseteq M(U)$ and $[\![0,1]\!]_{\frac{\pi}{n}, \frac{(n-1)\pi}{n}}=\frac{1}{2}+\frac{1}{2}i\tan(\frac{\pi}{n})$ by the intersection formula, it follows that $R+R\left(\frac{1}{2}+\frac{1}{2}i\tan(\frac{\pi}{n})\right) \subseteq M(U)$. 

Conversely, to show that $M(U) \subseteq R+R\left(\frac{1}{2}+\frac{1}{2}i\tan(\frac{\pi}{n})\right)$, it suffices to show that for any $p,q \in R+R\left(\frac{1}{2}+\frac{1}{2}i\tan(\frac{\pi}{n})\right)$, and any $\alpha, \beta \in U$, $[\![p,q]\!]_{\alpha, \beta} \in R+R\left(\frac{1}{2}+\frac{1}{2}i\tan(\frac{\pi}{n})\right)$. Without loss of generality, we may shift the points so that $p=0$. Any intersection obtained this way may also be obtained by replacing $q$ with its projection onto the real line using the same angle $\beta$ so we may assume $q \in R$.  Since $\alpha,\beta \in U$, we have $\alpha=\frac{(i+1)\pi}{n}$ and $\beta=\frac{(j+1)\pi}{n}$ for some $i,j \in \{0,1, 2, \dots, n-2\}$. Now, using the fact (twice) that $[\![p,q]\!]_{\alpha,\beta}=[\![0,q-p]\!]_{\alpha,\beta}$ yields 

\begin{align*}
[\![0,q]\!]_{\alpha, \beta}&= q[\![0,1]\!]_{\alpha,\beta}=q[\![x_0',x_{n-2}']\!]_{\frac{\pi}{n},\frac{(n-1)\pi}{n}}\\
&=q\left(x_0'+[\![0,x_{n-2}'-x_0']\!]_{\frac{\pi}{n},\frac{(n-1)\pi}{n}}\right)\\
&=q\left(x_0'+(x_{n-2}'-x_0')[\![0,1]\!]_{\frac{\pi}{n},\frac{(n-1)\pi}{n}}\right)\\
&=q\left(x_0'+(x_{n-2}'-x_0')\left(\frac{1}{2}+\frac{1}{2}\tan(\alpha)i\right)\right)
\end{align*}
where $x_0',x_{n-2}' \in R$ are as defined in Lemma \ref{R}. Finally, letting $a=qx_0'$ and $b=q(x_{n-2}'-x_0')$, we obtain $a+b(\frac{1}{2}+\frac{1}{2}\tan(\frac{\pi}{n})i \in R+R(\frac{1}{2}+\frac{1}{2}i\tan(\frac{\pi}{n})$ since $a,b \in R$. Therefore, $M(U)=R+R\left(\frac{1}{2}+\frac{1}{2}i\tan(\frac{\pi}{n})\right)$.

\end{proof}

Finally, to show that reflections of order $2n$ cannot be obtained if $n$ is even, it suffices to show that $\cos(\frac{\pi}{n})+i\sin(\frac{\pi}{n}) \notin M(U)$ if $n$ is even.

\begin{lemma} \label{inR}
Let $U$ be an angle set containing $\{0, \frac{\pi}{n}, \frac{2\pi}{n}, \dots, \frac{(n-1)\pi}{n}\}$. Then, $\cos(\frac{\pi}{n})+i\sin(\frac{\pi}{n}) \in M(U)$ if and only if $2\cos(\frac{\pi}{n}) \in M(U) \cap \R$.

\end{lemma}

\begin{proof}
First suppose $\cos\left(\frac{\pi}{n}\right)+i\sin\left(\frac{\pi}{n}\right)\}) \in M(U)$. Then $[\![0,\cos\left(\frac{\pi}{n}\right)+i\sin\left(\frac{\pi}{n} \right)]\!]_{0,\frac{2\pi}{n}}=2\cos\left(\frac{\pi}{n}\right)$ by the intersection formula, so $2\cos\left(\frac{\pi}{n}\right) \in M(U) \cap \R$. 

Conversely, if $2\cos\left(\frac{\pi}{n}\right) \in M(U) \cap \R$, then $[\![0,2\cos\left(\frac{\pi}{n}\right)]\!]_{\frac{\pi}{n},\frac{2\pi}{n}}=\cos\left(\frac{\pi}{n}\right)+i\sin\left(\frac{\pi}{n}\right)$ which then must be in $M(U)$.

\end{proof}

\begin{lemma} \label{evenodd}
Let $U=\{0,\frac{\pi}{n}, \frac{2\pi}{n}, \dots, \frac{(n-1)\pi}{n}\}$. Then, $2\cos(\frac{\pi}{n}) \in M(U)$ if and only if $n$ is odd.
\end{lemma}

\begin{proof}
If $n$ is odd, then $\cos(\frac{\pi}{n})+i\sin(\frac{\pi}{n}) \in M(U)$ as shown in Theorem \ref{nodd}. Then, by Lemma \ref{inR}, $2\cos(\frac{\pi}{n}) \in R$. 

Conversely, suppose that $n$ is even. Using results from \cite{WZ}, the minimal polynomial of $\cos(\frac{2\pi}{n})$ over $\mathbb{Q}$ has degree $k=\frac{\phi(n)}{2}$, where $\phi(n)$ is the number of positive integers less than $n$ that are relatively prime to $n$. Similarly, the minimal polynomial of $\cos(\frac{\pi}{n})=\cos(\frac{2\pi}{2n})$ has degree $\frac{\phi(2n)}{2}=\frac{2\phi(n)}{2}=\phi(n)=2k$ is even since $n>2$. Suppose that $f(x)$ is the minimal polynomial of $\cos(\frac{2\pi}{n})$, which has degree $k$ by \cite{WZ}. Then using the fact that $\cos(\frac{2\pi}{n})=2\cos^2(\frac{\pi}{n})-1$, it must be true that $\cos(\frac{\pi}{n})$ satisfies the polynomial $f(2x^2-1)$ which has degree $2k$, and contains only even powers of $x$. Since the minimal polynomial of $\cos\left(\frac{\pi}{n}\right)$ has degree $2k$, it must be true that its minimal polynomial is $f(2x^2-1)$ up to scaling. 

If $2\cos\left(\frac{\pi}{n}\right) \in R$ then we would have $2\cos(\frac{\pi}{n})=\frac{g(\cos(\frac{\pi}{n}))}{h(\frac{\cos(\pi}{n}))}$ where $g(x)$ and $h(x)$ contain only even powers of $x$ and both have degree at most $2k$ by Lemma \ref{M(U)}. We would then have $2\cos(\frac{\pi}{n})h(\cos(\frac{\pi}{n}))-g(\cos(\frac{\pi}{n}))=0$ which would mean there exists another minimal polynomial of $\cos(\frac{\pi}{n})$ containing some odd powers of $\cos(\frac{\pi}{n})$ which is a contradiction. 

    
\end{proof}

\begin{theorem} \label{neven}
    If $n \neq 2$ is even and $U=\{0,\frac{\pi}{n}, \frac{2\pi}{n}, \dots, \frac{(n-1)\pi}{n}\}$, then $P(U)\cong D_{n}$, the dihedral group of a regular $n$-gon.
\end{theorem}

\begin{proof} 
Suppose $n \neq 2$ is even and $U=\{0,\frac{\pi}{n}, \frac{2\pi}{n}, \dots, \frac{(n-1)\pi}{n}\}$. Since $U$ contains $n$ angles, the highest rotation order must be at most $2n$. Let $\theta=\frac{\pi}{n}$. By the intersection formula, $[\![0,1]\!]_{2\theta, \frac{\pi}{n}\left(\frac{n}{2}-2\right)}=\cos\left(\frac{2\pi}{n}\right)+i\sin\left(\frac{2\pi}{n}\right)$. 
Therefore, by Proposition \ref{rot-ref}, Since $U$ contains $\cos(2\theta)+i\sin(2\theta) \in M(U)$, and for every $\alpha \in U$, $\alpha+2\theta \in U$, $P(U)$ contains rotations by $2\theta$ which have order $\frac{2\pi}{\theta}=n$. By Lemma \ref{evenodd}, since $n$ is even, $2\cos\left(\frac{\pi}{n}\right) \notin M(U)$. Then, by Lemma \ref{inR}, $\cos\left(\frac{\pi}{n}\right)+i\sin\left(\frac{\pi}{n}\right) \notin M(U)$, so by Proposition \ref{rot-ref}, $P(U)$ does not contain rotations by $\frac{\pi}{n}$ which would have order $2n$. Therefore, $n$ is the highest rotation order of $P(U)$.

To show that $P(U)$ contains reflections across the angle $\theta$, it has already been noted that $\cos(2\theta)+i\sin(2\theta) \in M(U)$. It is also true by the definition of $U$ that for every $\alpha \in U$, $2\theta-\alpha \in U$ since $U$ consists of all multiples of $\theta$. Therefore, by Proposition \ref{rot-ref}, $P(U)$ contains a reflection across the angle $\frac{\pi}{n}$ and hence contains reflections.



Since $P(U)$ contains a rotation by an angle of $\frac{\pi}{n}$ and a reflections, all other symmetries in $D_n$ can be written as a composition of those so $P(U)$ is isomorphic to $D_n$ by Theorem \ref{mainthm2}.
\end{proof}

\begin{example} \label{p4m}
    Let $U=\{0, \frac{\pi}{4}, \frac{\pi}{2}, \frac{3\pi}{4}\}$. Using techniques from \cite{BR}, the origami set $M(U)$ is given by $M(U)=R+R i$ where $R=\Z[\frac{1}{2}]$. The group $P(U)$ contains rotations of order 4 as well as four reflection axes: $\frac{\alpha}{8}$, $\frac{3\alpha}{8}$, $\frac{5\alpha}{8}$, and $\frac{7\alpha}{8}$. Thus, $P(U)$ contains the following characteristics: rotations of order 4 and reflection axes of $\frac{\pi}{4}$ radians.

\end{example}

It has now been shown that the point group of an origami construction is either isomorphic to $D_n$ or $\Z/n\Z$ for an even integer $n$. In fact, it is true that every such group is the point group of the origami construction for some set of angles. We conclude with a theorem summarizing this.

\begin{theorem} \label{final}
Let $G$ be a finite group. Then, $G\cong P(U)$ for some angle set $U$ containing at least 3 angles if and only if $G$ is isomorphic to one of the following groups:
\begin{enumerate}
    \item $\Z/2\Z \times \Z/2\Z$
    \item $\Z/n\Z$
    \item $D_n$
\end{enumerate}
for an even integer $n$.
\end{theorem}
\begin{proof}
    It has already been shown in Theorems \ref{nodd}  and \ref{neven} that if $U$ contains more than three angles then $P(U)$ is isomorphic to one of $\Z/2\Z \times \Z/2\Z$, $\Z/n\Z$, or $D_n$ for some even integer $n$. It has also already been shown that if $U$ contains exactly three angles then $G(U)$ is isomorphic to the wallpaper group corresponding to either $p2$ so that $P(U)\cong \Z/2\Z$, $cmm$ so that $P(U) \cong \Z/2\Z \times \Z/2\Z$, or $p6m$ so that $P(U) \cong D_6$ by Theorem \ref{mainthm}. It remains to show the converse.
    
    Case 1: $G \cong \Z/2\Z \times \Z/2\Z$: Let $U=\{0,\frac{\pi}{4}, \frac{\pi}{2}\}$ so $G(U)$ is the wallpaper group $cmm$ whose point group $P(U)$ is isomorphic to $\Z/2\Z \times \Z/2\Z$.

   Case 2: $G\ \cong \Z/n\Z$ for some even integer $n$: Let $\alpha$ be an angle such that $0<\alpha < \frac{\pi}{2n}$. If $U=\{0, \alpha, \frac{\pi}{n}, \frac{2\pi}{n}, \frac{2\pi}{n}+\alpha, \dots, \frac{(n-1)\pi}{n}\}$ (where $\alpha+\frac{k\pi}{n} \in U $ if and only if $k$ is even) then $P(U) \cong \Z/n\Z$. This can be seen as follows: note that for every $\theta \in U$, $\theta+\frac{2\pi}{n} \in U$, $\frac{2\pi}{n}$ is the smallest angle for which this is true. Moreover, since $[\![0,1]\!]_{\frac{2\pi}{n}, \frac{\pi}{n}\left(\frac{n}{2}-2\right)}=\cos(\frac{2\pi}{n})+i\sin(\frac{2\pi}{n})$ so by Theorem \ref{mainthm2}, the highest rotation order is $2\pi/(\frac{2\pi}{n})=n$. 
    Note that $P(U)$ cannot contain reflections since there would have to be a reflection axis $\theta$ such that $0 \leq \theta < \frac{\pi}{n}$ so $0 \leq 2\theta < \frac{2\pi}{n}$. This cannot be true since that would mean that $2\theta \in U$ meaning that $2\theta=0$, $2\theta=\alpha$, or $2\theta=\frac{\pi}{n}$. If $2\theta=0$ then $2\theta-\alpha=-\alpha$ which is not in $U$. If $2\theta=\alpha$ then $2\theta-\frac{\pi}{2}=\alpha-\frac{\pi}{2}$ which again is also not in $U$. If $2\theta=\frac{\pi}{n}$ then $2\theta-\alpha=\frac{\pi}{n}-\alpha$ which is again not in $U$. Therefore, by Proposition \ref{rot-ref}, $P(U)$ does not contain any reflections so we must have $P(U) \cong \Z/n\Z$.
    
    Case 3: $G \cong D_n$ for some even integer $n$: By Theorem \ref{neven}, if $U=\{0, \frac{\pi}{n}, \frac{2\pi}{n}, \dots, \frac{(n-1)\pi}{n}\}$ then $P(U) \cong D_n$.
\end{proof}

This final theorem completely classifies all of the point groups for origami structures and relates them to three groups that are commonly found in algebra. The following example illustrates how Theorem \ref{final} can be applied to determine the characteristics of a structure and its isomorphism with the aforementioned minimum of 3 angles. 

\begin{example} \label{p6}
    Let $U=\{0, \frac{\pi}{4}, \frac{\pi}{3}, \frac{7\pi}{12}, \frac{2\pi}{3}, \frac{11\pi}{12}\}$. Using techniques from \cite{BR}, $M(U)=R+R \tan(\alpha)i$ where $R=\Z[\sqrt[3]{4}, \frac{1}{\sqrt[3]{4}}, \frac{1}{1+\sqrt[3]{4}}]$. Note the use of the fact that $\cos^2(\alpha)=\frac{1}{1+\tan^2(\alpha)}$. The group $P(U)$ contains rotations of order 6 but no reflection axes so $P(U) \cong \Z/6\Z$.
\end{example}






\bigskip




\begin{thebibliography}{100}

\bibitem{Armstrong} Armstrong, M. A. (1988). ``Groups and Symmetry." Springer.

\bibitem{BR} Jackson Bahr and Arielle Roth, ``Subrings of $\C$ Generated by Angles," \textit{Rose-Hulman Undergraduate Mathematics Journal}: Vol. 17: Iss. 1, Article 2, 2016.

\bibitem{Blumenson} L.E. Blumenson, ``A derivation of $n$-dimensional spherical coordinates." \textit{The American Mathematical Monthly}, Vol. 67, No. 1, pp. 63--66, 1960. 

\bibitem{BBDLG} Joseph Buhler, Steve Butler, Warwick De Launey and Ronald Graham, ``Origami rings." \textit{Journal of the Australian Mathematical Society} 92(3): 299-311, 2012.

\bibitem{REU} Yu X. Hong. ``Exploring Origami Generated Structures in $\mathbb{C}$," \url{http://dx.doi.org/10.1115/1.4034299}, 2016.

\bibitem{KS} J\"{u}rgen Kritschgau and Adriana Salerno, ``Origami constructions of the ring of integers of an imaginary quadratic field", INTEGERS: The Combinatorial Journal of Number Theory,  Vol. 17, \#A34, 2017.


\bibitem{Moller} Florian M\"{o}ller, ``When is an origami set a ring?", \url{https://arxiv.org/abs/1804.10449}, 2018.

\bibitem{Munkres} James R. Munkres, \emph{Topology: a first course}, Prentice Hall, Englewood Cliffs, New Jersey, 1975. 




\bibitem{Nedrenco} Dmitri Nedrenco, ``On Origami Rings", \url{https://arxiv.org/abs/1502.07995}, 2015.


\bibitem{Pawelec} Ewa Pawelec, et. al. ``The three-term recursion for Chebyshev polynomials
is mixed forward-backward stable", Numer Algor, Vol. 69, pp. 785-794, 2015.


\bibitem{Sasse} Vivek Sasse, ``Classification of the 17 Wallpaper Groups", \url{http://math.uchicago.edu/~may/REU2020/REUPapers/Sasse.pdf}, Chicago, Illinois, 2020 

\bibitem{WZ} William Watkins and Joel Zeitlin, ``The Minimal Polynomial of $\cos(\frac{2\pi}{n})$", The American Mathematical Monthly, Vol. 100, No. 5, pp. 471-474, 1993.

\end{thebibliography}
\end{document}